\documentclass[letterpaper, 10 pt, conference]{ieeeconf} 
\IEEEoverridecommandlockouts
\usepackage{cite}
\usepackage{amsmath,amssymb,amsfonts,bm}
\usepackage{algorithmic}
\usepackage{graphicx}
\usepackage{subcaption}
\usepackage{textcomp}
\usepackage{multirow}
\usepackage{xcolor}
\usepackage[ruled,vlined,onelanguage]{algorithm2e}
\usepackage{array}
\usepackage{fancyhdr}

\newcolumntype{C}[1]{>{\centering\arraybackslash}p{#1}}
\newcommand{\cov}[2]{\mathrm{cov}\left( #1, #2\right)}
\newcommand{\vecs}[1]{\mathrm{vec}\left( #1 \right)}
\newcommand{\vol}[1]{\mathrm{vol}\left( #1 \right)}
\newcommand{\E}[1]{\mathrm{E}\left\{#1\right\}}
\newcommand{\mat}[1]{\left[\begin{matrix} #1 \end{matrix}\right]}
\newcommand{\norm}[1]{\left\lVert#1\right\rVert}

\DeclareMathOperator{\ISOC}{ISOC}
\DeclareMathOperator{\OPT}{OPT}
\DeclareMathOperator{\RoA}{RoA}
\DeclareMathOperator{\comp}{comp}
\DeclareMathOperator{\mean}{mean}
\DeclareMathOperator{\const}{const}

\usepackage{ifthen}
\newboolean{showcomments}
\setboolean{showcomments}{false} 
\ifthenelse{\boolean{showcomments}}{ 
	\newcommand\td[1]{\textcolor{red}{\textbf{TODO:} #1}} 
	\newcommand\tdd[1]{} 
	\newcommand\comment[1]{\textcolor{blue}{\textbf{Comment:} #1}} 
	\newcommand\commentd[1]{} 
	\newcommand\frage[1]{\textcolor{orange}{\textbf{Rückfrage:} #1}} 
	\newcommand\fraged[1]{}
}{ 
	\newcommand\td[1]{}
	\newcommand\tdd[1]{} 
	\newcommand\comment[1]{} 
	\newcommand\commentd[1]{} 
	\newcommand\frage[1]{}
	\newcommand\fraged[1]{}
} 

\newtheorem{definition}{Definition}
\newtheorem{problem}{Problem}
\newtheorem{lemma}{Lemma}
\newtheorem{theorem}{Theorem}
\newtheorem{remark}{Remark}
\newtheorem{assumption}{Assumption}
\newtheorem{corollary}{Corollary}

\title{\LARGE \bf
	Bi-Level-Based Inverse Stochastic Optimal Control
}

\author{Philipp Karg$^{1}$, Manuel Hess$^{1}$, Balint Varga$^{1}$ and Sören Hohmann$^{1}$
	\thanks{$^{1}$All authors are with the Institute of Control Systems (IRS),
		Karlsruhe Institute of Technology (KIT), 76131 Karlsruhe, Germany. Corresponding author is Philipp Karg, {\tt\small philipp.karg@kit.edu}.}%
}

\begin{document}

\maketitle
\thispagestyle{fancy}
\pagestyle{fancy}

\fancyhf{}
\fancyhead[CO,CE]{\copyright 2024 EUCA. This paper has been accepted for publication and presentation at the 2024 European Control Conference.}

\begin{abstract}
	In this paper, we propose a new algorithm to solve the Inverse Stochastic Optimal Control (ISOC) problem of the linear-quadratic sensorimotor (LQS) control model. The LQS model represents the current state-of-the-art in describing goal-directed human movements. The ISOC problem aims at determining the cost function and noise scaling matrices of the LQS model from measurement data since both parameter types influence the statistical moments predicted by the model and are unknown in practice. We prove global convergence for our new algorithm and at a numerical example, validate the theoretical assumptions of our method. By comprehensive simulations, the influence of the tuning parameters of our algorithm on convergence behavior and computation time is analyzed. The new algorithm computes ISOC solutions nearly $33$ times faster than the single previously existing ISOC algorithm.
\end{abstract}


\section{Introduction} \label{sec:introduction}

Inverse Optimal Control (IOC) methods \cite{Molloy.2022} have gained significant research interest in the last years, from a theoretical- as well as application-oriented perspective. IOC methods answer the inverse question of OC problems: when are observed trajectories or control laws optimal, i.e. the result of an OC solution. Hence, IOC approaches aim at identifying the unknown cost function from observed trajectories, the so-called ground truth (GT) data, that are assumed to be optimal. An important application of IOC methods is the identification of goal-directed human movements. While there are strong indications (see e.g. \cite{Todorov.2002,Uno.1989}) that goal-directed human movements can be described by an OC model, the corresponding cost function is unknown in practice and needs to be identified from measurement data to verify the model hypothesis or to use the identified model for the design of human-machine systems (e.g. prediction, classification or support of human movements\comment{imitation possible as well}). For deterministic OC models (see e.g. \cite{Uno.1989}), IOC methods are applied with these goals for example in \cite{Jin.2019,Oguz.2018,Berret.2011}. However, deterministic OC models describe human movements by the so-called feedforward planning approach \cite{Gallivan.2018}. Via the OC model only the process of specifying a desired trajectory by the human is described. This movement planning stage follows a separated execution stage where the desired trajectory is tracked. Hence, sensory feedback on the movement trajectory is not taken into account adequately. Currently, feedback control approaches \cite{Gallivan.2018} based on stochastic optimal control (SOC) models are more promising since they consider the continuous sensory feedback of the human on its movement trajectories and are able to explain the stochastic nature of them \cite{Gallivan.2018,Todorov.2002}. The current main representative of these approaches is the linear-quadratic sensorimotor (LQS) control model, which is derived from the linear-quadratic Gaussian (LQG) model by adding a control-dependent noise process to the state and a state-dependent noise process to the output equation. Both extensions are used to reflect the special characteristics of the stochastic nature of human movements, i.e. faster movements are performed and perceived more inaccurately \cite{Todorov.2005,Todorov.2002}\footnote{With the LQS model, the human hand is typically modeled as point mass since the biomechanics of the complete human arm are nonlinear in general.}.

As for deterministic OC models, for SOC models an identification method is needed to determine their unknown parameters from measurement/GT data in order to verify the model hypothesis in a specific task in a data-driven manner \cite{Karg.2023b} or to use the identified models for the design of human-machine systems. Until very recently \cite{Karg.2023a}, an approach that solves this Inverse Stochastic Optimal Control (ISOC) problem for the LQS model was missing and only impractical special cases, like the LQG model \cite{Priess.2014} or well-selected parameters of the LQS model \cite{Kolekar.2018}, were considered so far. For the LQS model, the ISOC problem consists of identifying the cost function and noise scaling matrices since both parameter types influence the predicted statistical moments of the system quantities. However, our previous method \cite{Karg.2023a} still shows high computation times (nearly one day on a standard PC) and lacks a guaranteed, proved global convergence behavior. Both challenges are solved in this paper by proposing a new ISOC algorithm with proved global convergence, which solves the ISOC problem in a simulation example nearly $33$ times faster than our previous method \cite{Karg.2023a}. Based on the simulation example, we furthermore validate the theoretical assumptions of our new approach and analyze the influence of its tuning parameters on convergence behavior and computation time.


\section{Problem Definition} \label{sec:problem_definition}


As explained in Section~\ref{sec:introduction}, when describing the human biomechanics with a state equation, we need to consider a control-dependent noise process $\{\sum_{i=1}^{c} \sigma_i^{\bm{u}} \varepsilon_t^{(i)}  \bm{B} \bm{F}_i \bm{u}_t\}$:
\begin{align} \label{eq:state_equation}
	\bm{x}_{t+1} = \bm{A}\bm{x}_t + \bm{B}\bm{u}_t + \bm{\Sigma}^{\bm{\alpha}}\bm{\alpha}_t + \sum_{i=1}^{c} \sigma_i^{\bm{u}} \varepsilon_t^{(i)}  \bm{B} \bm{F}_i \bm{u}_t,
\end{align}
where $\bm{x}\in\mathbb{R}^n$ denotes the system state, $\bm{u}\in\mathbb{R}^m$ the control variable, $\{\bm{\alpha}_t\}$ a standard white Gaussian noise process in sample space $\mathbb{R}^p$ and $\{\bm{\varepsilon}_t\}$ ($\bm{\varepsilon}^\intercal_t = \mat{\varepsilon_t^{(1)} & \dots & \varepsilon_t^{(c)}}$) a standard white Gaussian noise process in sample space $\mathbb{R}^c$. Furthermore, $\bm{A}$, $\bm{B}$, $\bm{\Sigma}^{\bm{\alpha}}$ and $\bm{F}_i$ are matrices of appropriate dimension, where $\bm{A}$ and $\bm{B}$ may depend on time. The random variables $\varepsilon_t^{(i)}$ are scaled with a constant $\sigma_i^{\bm{u}} \in \mathbb{R}$ and with $\bm{u}_t$ \cite{Todorov.2005}. The stochastic process $\{\bm{x}_t\}$ is initialized with $\E{\bm{x}_0}$ and $\cov{\bm{x}_0}{\bm{x}_0} = \bm{\Omega}_0^{\bm{x}}$.
Human perception is described by an output equation with a state-dependent noise process $\{\sum_{i=1}^{d} \sigma_i^{\bm{x}} \epsilon_t^{(i)} \bm{H} \bm{G}_i \bm{x}_t\}$ in the LQS model:
\begin{align} \label{eq:output_equation}
	\bm{y}_t = \bm{H}\bm{x}_t + \bm{\Sigma}^{\bm{\beta}}\bm{\beta}_t + \sum_{i=1}^{d} \sigma_i^{\bm{x}} \epsilon_t^{(i)} \bm{H} \bm{G}_i \bm{x}_t,
\end{align}
where $\bm{y}\in\mathbb{R}^r$ denotes the observed output, $\{\bm{\beta}_t\}$ a standard white Gaussian noise process in sample space $\mathbb{R}^q$ and $\{\bm{\epsilon}_t\}$ ($\bm{\epsilon}^\intercal_t = \mat{\epsilon_t^{(1)} & \dots & \epsilon_t^{(d)}}$) a standard white Gaussian noise process in sample space $\mathbb{R}^d$. Moreover, $\bm{H}$, $\bm{\Sigma}^{\bm{\beta}}$ and $\bm{G}_i$ are matrices of appropriate dimension, where $\bm{H}$ may depend on time. Similarly to \eqref{eq:state_equation}, the random variables $\epsilon_t^{(i)}$ are scaled by a constant $\sigma_i^{\bm{x}} \in \mathbb{R}$ and $\bm{x}_t$ \cite{Todorov.2005}.
The random variables $\bm{x}_t$, $\bm{\alpha}_t$, $\bm{\varepsilon}_t$, $\bm{\beta}_t$ and $\bm{\epsilon}_t$ at time $t$ are assumed to be independent to each other.
The performance criterion which drives the goal-directed human movements is given by 
\begin{align} \label{eq:cost_function}
	J = \E{\bm{x}^\intercal_N \bm{Q}_N \bm{x}_N + \sum_{t=0}^{N-1} \bm{x}_t^\intercal \bm{Q}_t \bm{x}_t + \bm{u}_t^\intercal \bm{R} \bm{u}_t},
\end{align}
where $\bm{Q}_N = \sum_{i=1}^{S_N} s_{N,i}\bm{q}_{N,i}\bm{q}_{N,i}^\intercal$ ($\bm{q}_{N,i} \in \mathbb{R}^n$, $s_{N,i} \in \mathbb{R}$), $\bm{Q}_t = \sum_{i=1}^{S_Q} s_{Q,i}\bm{q}_{Q,t,i}\bm{q}_{Q,t,i}^\intercal$ ($\bm{q}_{Q,t,i} \in \mathbb{R}^n$, $s_{Q,i} \in \mathbb{R}$) and $\bm{R} = \sum_{i=1}^{S_R} s_{R,i}\bm{q}_{R,i}\bm{q}_{R,i}^\intercal$ ($\bm{q}_{R,i} \in \mathbb{R}^m$, $s_{R,i} \in \mathbb{R}$).

Solving the LQS optimal control Problem~\ref{problem:SOC} models the optimization performed by the human in its perception-action cycle.
\begin{problem} \label{problem:SOC}
	Find an admissible control strategy $\bm{u}_t = \bm{\pi}_t(\bm{u}_0,\dots,\bm{u}_{t-1},\bm{y}_0,\dots,\bm{y}_{t-1})$ for the system defined by \eqref{eq:state_equation} and \eqref{eq:output_equation} such that \eqref{eq:cost_function} is minimal.
\end{problem}

Due to the multiplicative noise processes in \eqref{eq:state_equation} and \eqref{eq:output_equation}, the separation theorem does not hold \cite{Liang.2023,Todorov.2005,Moore.1999,Joshi.1976}. Following the procedure proposed in \cite{Todorov.2005}, Lemma~\ref{lemma:solution_SOC} shows the solution of Problem~\ref{problem:SOC} considered throughout the paper which takes the interdependence between controller and filter into account.
\begin{lemma} \label{lemma:solution_SOC}
	Let $\bm{R}$ and $\bm{\Sigma}^{\bm{\beta}}{\bm{\Sigma}^{\bm{\beta}}}^\intercal$ be positive definite. Furthermore, let the history of control and output values for the admissible control strategies $\bm{u}_t = \bm{\pi}_t(\bm{u}_0,\dots,\bm{u}_{t-1},\bm{y}_0,\dots,\bm{y}_{t-1})$ (cf.~Problem~\ref{problem:SOC}) be represented by the estimation $\hat{\bm{x}}_t$ of a linear filter:
	\begin{align} \label{eq:assm_linear_filter}
		\hat{\bm{x}}_{t+1} = \bm{A}\hat{\bm{x}}_t + \bm{B}\bm{u}_t + \bm{K}_t \left( \bm{y}_t - \bm{H}\hat{\bm{x}}_t \right) + \bm{\Sigma}^{\bm{\gamma}}\bm{\gamma}_t,
	\end{align}
	where $\bm{K}_t$ ($\forall t \in \{0,\dots,N-1\}$) are constant filter matrices of appropriate dimension and $\{\bm{\gamma}_t\}$\footnote{The random variables $\bm{\gamma}_t$ are independent to $\bm{x}_t$, $\bm{\alpha}_t$, $\bm{\varepsilon}_t$, $\bm{\beta}_t$ and $\bm{\epsilon}_t$.} a standard white Gaussian noise process in $\mathbb{R}^{l}$ scaled by a constant matrix $\bm{\Sigma}^{\bm{\gamma}}$ of appropriate dimension.
	Then, the optimal control law is given by $\bm{u}_t = \bm{\pi}_t(\hat{\bm{x}}_t) = - \bm{L}_t \hat{\bm{x}}_t$ with\footnote{Recursive formulas for $\bm{Z}_t^{\bm{x}}$ and $\bm{Z}_t^{\bm{e}}$ can be found in \cite{Todorov.2005}.}
	\begin{align}
		\bm{L}_{t} &= \Big(\bm{R} + \bm{B}^{\intercal} \bm{Z}_{t+1}^{\bm{x}} \bm{B} \nonumber\\
		&\hphantom{=} + \sum_{i} (\sigma_i^{\bm{u}})^2 \bm{F}_i^{\intercal} \bm{B}^{\intercal} \left( \bm{Z}_{t+1}^{\bm{x}} \! + \! \bm{Z}_{t+1}^{\bm{e}} \right) \bm{B} \bm{F}_i \Big)^{-1} \bm{B}^{\intercal} \bm{Z}_{t+1}^{\bm{x}} \bm{A} \label{eq:SensoControlLawL}. 
	\end{align}
	Moreover, the optimal constant filter matrices $\bm{K}_t$ are given by\footnote{Recursive formulas for $\bm{P}_t^{\hat{\bm{x}}}$, $\bm{P}_t^{\bm{e}}$ and $\bm{P}_t^{\hat{\bm{x}}\bm{e}} = \bm{P}_t^{\bm{e}\hat{\bm{x}}}$ can be found in \cite{Todorov.2005}.}
	\begin{align}
		\bm{K}_{t} &= \bm{A} \bm{P}_{t}^{\bm{e}} \bm{H}^{\intercal} \Big( \bm{H} \bm{P}_{t}^{\bm{e}} \bm{H}^{\intercal} + \bm{\Sigma}^{\bm{\beta}}{\bm{\Sigma}^{\bm{\beta}}}^{\intercal} \nonumber \\
		&\hphantom{=} + \sum_{i} (\sigma_i^{\bm{x}})^2\bm{H}\bm{G}_i \left( \bm{P}_t^{\bm{e}} \! + \! \bm{P}_t^{\hat{\bm{x}}} \! + \! \bm{P}_t^{\hat{\bm{x}}\bm{e}} \! + \! \bm{P}_t^{\bm{e}\hat{\bm{x}}} \right) \bm{G}_i^{\intercal}\bm{H}^{\intercal} \Big)^{-1} \label{eq:SensoEstimatorK}.
	\end{align}
\end{lemma}
\begin{proof}
	With the assumption on a linear filter providing a sufficient statistic for the history of control and output values, we can derive a Bellman equation based on \cite[Lemma~3.2, p.~261]{Astrom.1970}. Minimizing $J$~\eqref{eq:cost_function} w.r.t. the admissible control strategy $\bm{\pi}_0,\bm{\pi}_1,\dots$ yields
	\begin{align} \label{eq:solution_ISOC_proof_1}
		\min_{\bm{\pi}_0,\dots} J = \E{ \min_{\bm{\pi}_0,\dots} \E{ \E{ J \mid \bm{x}_0, \bm{\hat{x}}_0 } \mid \hat{\bm{x}}_0 } }.
	\end{align}
	By defining a value function
	\begin{align} \label{eq:solution_ISOC_proof_2}
		V_t(\bm{x}_t, \hat{\bm{x}}_t) &= \text{E} \biggl\{ \bm{x}^\intercal_N \bm{Q}_N \bm{x}_N \nonumber \\
			&\hphantom{=}+ \sum_{\tau=t}^{N-1} \bm{x}_\tau^\intercal \bm{Q}_\tau \bm{x}_\tau \! + \! \bm{u}_\tau^\intercal \bm{R} \bm{u}_\tau \mid \bm{x}_t, \hat{\bm{x}}_t \biggl\}
	\end{align} 
	and its optimized version $V_t^*(\hat{\bm{x}}_t) = \min_{\bm{\pi}_t,\dots}\E{V_t(\bm{x}_t, \hat{\bm{x}}_t)\mid \hat{\bm{x}}_t }$, from \eqref{eq:solution_ISOC_proof_1} the recursive Bellman equation
	\begin{align} \label{eq:solution_ISOC_proof_3}
		V_t(\bm{x}_t, \hat{\bm{x}}_t) &= \text{E} \{ \bm{x}^\intercal_t \bm{Q}_t \bm{x}_t \! + \! \bm{u}^\intercal_t \bm{R} \bm{u}_t \nonumber \\
		&\hphantom{=}+ V_{t+1}(\bm{x}_{t+1}, \hat{\bm{x}}_{t+1}) \mid \bm{x}_t, \hat{\bm{x}}_t \}
	\end{align}
	follows. Now, it can be shown by induction (backwards, starting with $t=N$) that the value function $V_t$ is quadratic, i.e. $V_t(\bm{x}_t, \hat{\bm{x}}_t) = \bm{x}_t^\intercal \bm{Z}_t^{\bm{x}} \bm{x}_t + \bm{e}_t^\intercal \bm{Z}_t^{\bm{e}} \bm{e}_t + z_t$ ($\bm{e}_t = \bm{x}_t - \hat{\bm{x}}_t$) \cite{Todorov.2005}. Moreover, by computing $V_t^*(\hat{\bm{x}}_t)$ from $V_t(\bm{x}_t, \hat{\bm{x}}_t)$, the linear control law $\bm{\pi}_t(\hat{\bm{x}}_t) = -\bm{L}_t\hat{\bm{x}}_t$ with \eqref{eq:SensoControlLawL} results. The calculation of $V_t^*$ and $\bm{\pi}_t$ is performed using \eqref{eq:solution_ISOC_proof_3} and assuming a given, fixed sequence of filter gains $\bm{K}_t$. In order to find the optimal filter matrices, $J$ is minimized w.r.t. to $\bm{K}_0, \bm{K}_1, \dots$ with assumed fixed control matrices $\bm{L}_t$. This can be done via forward recursion $\bm{K}_t = \arg \min_{\bm{K}_t}\E{V_{t+1}^*(\hat{\bm{x}}_{t+1})}$ ($\forall t \in \{0,\dots,N-1\}$), which yields \eqref{eq:SensoEstimatorK}.
\end{proof}

Since $\bm{L}_t$ is derived under the assumption of specific filter gains $\bm{K}_t$ and the filter gains vice versa, their calculation is done via an iterative procedure starting with a guess for $\bm{K}_t$. This leads to an alternating optimization (AO) \cite{Bezdek.2002}. In each AO optimization (with respect to $\bm{L}_t$ or $\bm{K}_t$) a unique global optimizer exists due to $\bm{R}$ and $\bm{\Sigma}^{\bm{\beta}}{\bm{\Sigma}^{\bm{\beta}}}^\intercal$ being positive definite. If in one AO optimization step the global optimal solution for $\bm{L}_t$ or $\bm{K}_t$ results, global convergence for the other parameter type follows \cite[Theorem~2]{Bezdek.2002}. This behavior is guaranteed e.g. in the LQG case (separation theorem). 

The statistical moments of $\{\bm{x}_t\}$ according to Theorem~\ref{theorem:mean_covariance_SOC} characterize the predictions of the LQS model. 
\begin{theorem} \label{theorem:mean_covariance_SOC}
	Applying the optimal control strategy according to Lemma~\ref{lemma:solution_SOC} to the system defined by \eqref{eq:state_equation} and \eqref{eq:output_equation} leads to the mean $\E{\bm{x}_t}$ and covariance $\bm{\Omega}_t^{\bm{x}} = \cov{\bm{x}_t}{\bm{x}_t}$:
	\begin{align}
		\mat{\E{\bm{x}_{t+1}} \\ \E{\hat{\bm{x}}_{t+1}}} &= \bm{\mathcal{A}}_t \mat{\E{\bm{x}_{t}} \\ \E{\hat{\bm{x}}_{t}}} \label{eq:SensoSolutionEW}, \\
		\mat{\bm{\Omega}_{t+1}^{\bm{x}} \!\!& \bm{\Omega}_{t+1}^{\bm{x}\hat{\bm{x}}} \\ \bm{\Omega}_{t+1}^{\hat{\bm{x}}\bm{x}} \!\!& \bm{\Omega}_{t+1}^{\hat{\bm{x}}}} &= \bm{\mathcal{A}}_t \mat{\bm{\Omega}_{t}^{\bm{x}} \!\!\!\!& \bm{\Omega}_{t}^{\bm{x}\hat{\bm{x}}} \\ \bm{\Omega}_{t}^{\hat{\bm{x}}\bm{x}} \!\!\!\!& \bm{\Omega}_{t}^{\hat{\bm{x}}}} \bm{\mathcal{A}}^\intercal_{t} \nonumber \\		
		&\hphantom{=} + \mat{\bm{\Omega}^{\bm{\xi}} \!\!\!\!& \bm{0} \\ \bm{0} \!\!\!\!& \bm{\Omega}^{\bm{\eta}} + \bm{K}_{t} \bm{\Omega}^{\bm{\omega}} \bm{K}^{\intercal}_{t}} + \mat{\bar{\bm{\Omega}}_t^{\hat{\bm{x}}} \!\!\!\!& \bm{0} \\ \bm{0} \!\!\!\!& \bar{\bm{\Omega}}_t^{\bm{x}}} \label{eq:SensoSolutionCOV},
	\end{align}	
	with $\bm{\Omega}_t^{\bm{\xi}} = \bm{\Sigma}^{\bm{\alpha}}{\bm{\Sigma}^{\bm{\alpha}}}^\intercal$, $\bm{\Omega}_t^{\bm{\eta}} = \bm{\Sigma}^{\bm{\gamma}}{\bm{\Sigma}^{\bm{\gamma}}}^\intercal$, $\bm{\Omega}_t^{\bm{\omega}} = \bm{\Sigma}^{\bm{\beta}}{\bm{\Sigma}^{\bm{\beta}}}^\intercal$, $\bar{\bm{\Omega}}_t^{\hat{\bm{x}}} = \sum_{i} (\sigma_i^{\bm{u}})^2 \bm{B}\bm{F}_i \bm{L}_t \left( \bm{\Omega}_{t}^{\hat{\bm{x}}} + \E{\hat{\bm{x}}_t}\E{\hat{\bm{x}}_t}^{\intercal} \right) \bm{L}_t^{\intercal} \bm{F}_i^{\intercal} \bm{B}^\intercal$, $\bar{\bm{\Omega}}_t^{\bm{x}} = \sum_{i} (\sigma_i^{\bm{x}})^2\bm{K}_t \bm{H}\bm{G}_i \left( \bm{\Omega}_{t}^{\bm{x}} + \E{\bm{x}_t}\E{\bm{x}_t}^{\intercal} \right) \bm{G}_i^{\intercal}\bm{H}^\intercal \bm{K}_t^{\intercal}$, 
	\begin{align} \label{eq:mathcalA}
		\bm{\mathcal{A}}_{t} = \begin{bmatrix} \bm{A} & -\bm{B} \bm{L}_{t} \\ \bm{K}_{t} \bm{H} & \bm{A} - \bm{K}_{t} \bm{H} - \bm{B} \bm{L}_{t} \end{bmatrix}
	\end{align}
	and $\E{\hat{\bm{x}}_0}=\E{\bm{x}_0}$, $\bm{\Omega}_t^{\hat{\bm{x}}\bm{x}}=\bm{\Omega}_t^{\bm{x}\hat{\bm{x}}}=\bm{\Omega}_t^{\hat{\bm{x}}}=\bm{0}$.
\end{theorem}
\begin{proof}
	See \cite[Lemma~2]{Karg.2023a}.
\end{proof}

Although Theorem~\ref{theorem:mean_covariance_SOC} gives recursive formulas to compute mean~$\E{\bm{x}_t}$ and covariance~$\bm{\Omega}_t^{\bm{x}}$ of the system state, i.e. in context of human movement modeling the average behavior and the variability patterns of the movements, unknown model parameters need to be identified from human measurement data before the model can be used. This leads to the inverse problem (cf.~Problem~\ref{problem:ISOC}) of the LQS optimal control Problem~\ref{problem:SOC}. From GT data (cf.~Definition~\ref{definition:GT_data}) that are assumed to be measured realizations of the stochastic processes of (in general only some) system states resulting in the LQS model with unknown GT parameters~$\bm{\theta}^*$ (cf.~Assumption~\ref{assumption:GT_data}), sufficient approximations of mean~$\hat{\bm{m}}_t$ and covariance values~$\hat{\bm{\Omega}}_t^{\bm{x}^*}$ are computed. Then, the ISOC problem aims at finding parameters~$\bm{\theta}^\bigtriangleup$ that yield a stochastic process $\{\bm{x}_t^{\bigtriangleup}\}$ with matching mean and covariance values to the GT data. Corollary~\ref{corollary:model_parameters} defines the parameter~$\bm{\theta}$ that influence the model predictions $\E{\bm{x}_t}$ and $\bm{\Omega}_t^{\bm{x}}$ and thus, need to be identified. Identifying cost function and noise scaling matrices from GT data distinguishes the ISOC problem from deterministic IOC problems, which only search for cost function parameters.
\begin{corollary} \label{corollary:model_parameters}
	Applying the optimal control strategy according to Lemma~\ref{lemma:solution_SOC} to the system defined by \eqref{eq:state_equation} and \eqref{eq:output_equation} leads to the mean $\E{\bm{x}_t}$~\eqref{eq:SensoSolutionEW} and covariance $\bm{\Omega}_t^{\bm{x}}$~\eqref{eq:SensoSolutionCOV}, which both depend on $\bm{\theta} = \mat{ \bm{s}^\intercal & \bm{\sigma}^\intercal }^\intercal \in \mathbb{R}^{\Theta}$ with $\bm{s}^\intercal = \mat{ s_{N,1} \!\! & \!\! \dots \!\! & \!\! s_{N,S_N} & s_{Q,1} \!\! & \!\! \dots \!\! & \!\! s_{Q,S_Q} & s_{R,1} \!\! & \!\! \dots \!\! & \!\! s_{R,S_R} }$ and $\bm{\sigma}^\intercal = \mat{ \vecs{\bm{\Sigma}^{\bm{\alpha}}}^\intercal & \vecs{\bm{\Sigma}^{\bm{\beta}}}^\intercal & \sigma_1^{\bm{u}} \!\! & \!\! \dots \!\! & \!\! \sigma_c^{\bm{u}} & \sigma_1^{\bm{x}} \!\! & \!\! \dots \!\! & \!\! \sigma_d^{\bm{x}} }$\footnote{We assume $\bm{\Sigma}^{\bm{\gamma}}=\bm{0}$ throughout the paper. Its consideration would be straightforward.}.
\end{corollary}
\begin{proof}
	The dependencies can directly be derived from Theorem~\ref{theorem:mean_covariance_SOC}: \eqref{eq:SensoSolutionEW} and \eqref{eq:SensoSolutionCOV} depend on $\bm{\mathcal{A}}_t$ which in turn depends on $\bm{L}_t$; then, $\bm{L}_t$ on $\bm{R}$ and $\sigma_i^{\bm{u}}$ and via $\bm{Z}_t^{\bm{x}}$ on $\bm{Q}_N$, $\bm{Q}_t$, $\sigma_i^{\bm{x}}$ and $\bm{K}_t$, which finally depends on $\bm{\Sigma}^{\bm{\beta}}$ and via $\bm{P}_t^{\bm{e}}$ on $\bm{\Sigma}^{\bm{\alpha}}$.
\end{proof}
\comment{LQG: noise scaling matrices of state and output equation both influence covariance prediction of the system state (see notes), mean only influenced by cost function matrices ($\bm{L}_t$ only dependent on them, separation theorem) - LQS: mean only influenced by $\bm{L}_t$ (but here also dependent on noise scaling parameters, also in the fully observable case), covariance influenced by all parameters (cost function and noise scaling matrices, especially also the noise scaling parameters of the output equation, although their influence is somehow reduced by the design procedure of the filter matrices $\bm{K}_t$, but their influence is not suppressed, they introduce more ambiguity) - LQS (fully observable): here, one can see (and possibly show) that the output noise processes introduce more ambiguity which is however not necessary for describing arbitrary observed moments - theoretical statements (thesis): show influence on mean and covariance with simplifications of formulas (see notes), discuss LQG case; show that the fully observable case is enough to model arbitrary mean and covariance curves}

\begin{definition} \label{definition:GT_data}
	The GT data are realizations of an unknown stochastic process $\{\bm{M}\bar{\bm{x}}_t\}$, where $\bar{\bm{x}} \in \mathbb{R}^{n}$ and the system state $\bm{x}$ in \eqref{eq:state_equation} represent the same physical quantities. Furthermore, $\bm{M} \in \mathbb{R}^{\bar{n}\times n}$ follows from the identity matrix by deleting rows that correspond to states/quantities that are not measured in the GT data. Finally, $\hat{\bm{m}}_t$ and $\hat{\bm{\Omega}}^{\bm{x}^*}_t$ are mean and covariance values computed from the realizations of $\{\bm{M}\bar{\bm{x}}_t\}$.
\end{definition}

\begin{assumption} \label{assumption:GT_data}
	The GT data are realizations of the stochastic process $\{\bm{M}\bm{x}_t^*\}$, where $\{\bm{x}_t^*\}$ results from the application of the admissible control strategy that is optimal w.r.t. the GT parameters $\bm{\theta}^*$ (cf.~Lemma~\ref{lemma:solution_SOC}) to the system \eqref{eq:state_equation}, \eqref{eq:output_equation}: $\hat{\bm{m}}_t \approx \E{\bm{M}\bm{x}^*_t}$ and $\hat{\bm{\Omega}}^{\bm{x}^*}_t \approx \bm{M} \bm{\Omega}^{\bm{x}^*}_t \bm{M}^\intercal$ hold.
\end{assumption}

\begin{assumption} \label{assumption:model_structure}
	The matrices $\bm{A}$, $\bm{B}$, $\bm{F}_i$, $\bm{H}$ and $\bm{G}_i$ are known as well as the vectors $\bm{q}_{N,i}$, $\bm{q}_{Q,t,i}$ and $\bm{q}_{R,i}$.
\end{assumption}

\begin{problem} \label{problem:ISOC}
	Find parameters~$\bm{\theta}^{\bigtriangleup}$ such that the stochastic process $\{\bm{x}_t^{\bigtriangleup}\}$, which results from the application of the admissible control strategy that is optimal w.r.t. the parameters $\bm{\theta}^{\bigtriangleup}$ (cf.~Lemma~\ref{lemma:solution_SOC}) to the system \eqref{eq:state_equation}, \eqref{eq:output_equation}, yields $\E{\bm{M}\bm{x}^{\bigtriangleup}_t} = \hat{\bm{m}}_t$ and $\bm{M} \bm{\Omega}^{\bm{x}^{\bigtriangleup}}_t \bm{M}^\intercal = \hat{\bm{\Omega}}^{\bm{x}^*}_t$.
\end{problem}

\begin{remark} \label{remark:nonGaussian_states}
	Although the stochastic processes $\{\bm{\alpha}_t\}$, $\{\bm{\varepsilon}_t\}$, $\{\bm{\beta}_t\}$, $\{\bm{\epsilon}_t\}$ and $\{\bm{\gamma}_t\}$ in \eqref{eq:state_equation}, \eqref{eq:output_equation} and \eqref{eq:assm_linear_filter} are Gaussian, $\{\bm{x}_t\}$ is not, due to the multiplicative noise processes in \eqref{eq:state_equation} and \eqref{eq:output_equation} (see also \cite{Joshi.1976}).
\end{remark}
\comment{with the statistical moments predicted by the theorem one can derive a Gaussian approximation for every point in time $t$; however, the correlation between different points in time (which exists!) cannot be}


\section{Bi-level-based Inverse Stochastic Optimal Control} \label{sec:bilevel_ISOC}

In this section, we propose our new ISOC algorithm to solve Problem~\ref{problem:ISOC} which is based on a bi-level structure. Bi-level-based algorithms are common to solve inverse problems of optimal control (see e.g. \cite{Karg.2023b,Karg.2023a,Oguz.2018,Berret.2011,Mombaur.2010}). In general, in such methods a parameter optimization problem is defined where the objective function quantifies how well the trajectories, which are optimal w.r.t. a current guess of the cost function and the noise scaling parameters, match the observed trajectories in the GT data. Since the optimization of this objective function requires for each function evaluation the solution of a SOC problem, a so-called bi-level structure results. 
Typically, it is unknown or need to be verified if the GT trajectories are indeed a SOC solution, i.e. Assumption~\ref{assumption:GT_data} cannot be guaranteed. In these situations, it is inevitable to have a bi-level-based ISOC approach, at least as baseline method for more advanced ones, since only bi-level-based algorithms with global convergence behavior guarantee the best possible parameters w.r.t. the chosen metric to define the objective function of the parameter optimization problem.
\comment{in more advanced/indirect methods for data-based IOC, correlation between residual and trajectory error only present when (impractical) assumptions are fulfilled, e.g. in Hamilton method assumption of known basis functions and rank condition + optimization of residual error can be numerically ill-balanced and leads to higher trajectory errors than directly optimizing the trajectory error} 

In Subsection~\ref{subsec:parameter_opt_problem}, we define the parameter optimization problem of our new ISOC algorithm. Then, we explain our new approach to solve it in Subsection~\ref{subsec:upper_level_opt_TRLwARoA} (so-called upper level optimization) and how we treat its non-convexity. In Subsection~\ref{subsec:lower_level_opt}, the solution to the SOC problem, which is a constraint of the parameter optimization problem defined in Subsection~\ref{subsec:parameter_opt_problem} and called lower level optimization, is discussed.

\subsection{Parameter Optimization Problem} \label{subsec:parameter_opt_problem}

Problem~\ref{problem:ISOC_parameter_opt} defines the parameter optimization derived from Problem~\ref{problem:ISOC}. The objective function $J_{\ISOC}(\bm{\theta},\hat{\bm{m}}_t,\hat{\bm{\Omega}}_t^{\bm{x}^*})$ is based on the variance accounted for (VAF) metric, successfully applied for model regression with real data \cite{Karg.2023b,Kolekar.2018}. The optimization w.r.t. $\bm{\theta}$ is performed on a feasible set~$\mathcal{U}$ to guarantee the existence of a well-defined SOC solution in the lower level constituted by constraints \eqref{eq:constraint_SOC} and \eqref{eq:constraint_mean_covariance}.
\begin{definition} \label{definition:feasible_set}
	The feasible set $\mathcal{U}$ is given by 
	\begin{align} \label{eq:feasible_set}
		\mathcal{U} = \{\bm{\theta} \in \mathbb{R}^{\Theta}: \theta_i \in [a_i, b_i], \forall i \in \{1,\dots,\Theta\}\},
	\end{align}
	where $a_i, b_i \in \mathbb{R}$ and $b_i > a_i$ ($\forall i \in \{1,\dots,\Theta\}$), such that for every $\bm{\theta} \in \mathcal{U}$ Lemma~\ref{lemma:solution_SOC} can be applied\comment{implicitly we assume that convergence is given here; however, in future, see comment before, we should be able to change this here to the positive definiteness assumption on the matrices only} and such that $\bm{\theta}^* \in \mathcal{U}$ if Assumption~\ref{assumption:GT_data} holds\comment{depending on the specific element of $\bm{\theta}$ scaled versions of $\bm{\theta}^*$ are possible, but not in general, see e.g. noise scalings}.
\end{definition}

\begin{problem} \label{problem:ISOC_parameter_opt}
	Find a global optimizer $\bm{\theta}^{\bigtriangleup}$ of
	\begin{subequations} 
		\label{eq:ISOC_parameter_opt}
		\begin{align} 
			\min_{\bm{\theta}} \hphantom{-} &J_{\ISOC}(\bm{\theta},\hat{\bm{m}}_t,\hat{\bm{\Omega}}_t^{\bm{x}^*}) = \nonumber \\
			\min_{\bm{\theta}} - &\frac{ \bm{w}_{\text{m}}^\intercal \bm{m}^{\text{VAF}}(\bm{\theta},\hat{\bm{m}}_t) + \bm{w}_{\text{v}}^\intercal \vecs{\bm{\Omega}^{\text{VAF}}(\bm{\theta},\hat{\bm{\Omega}}^{\bm{x}^*}_t)} }{ \norm{\bm{w}_{\text{m}}}_1 + \norm{\bm{w}_{\text{v}}}_1 } \label{eq:J_ISOC} \\
			\text{s.t.} \enspace &\bm{\theta} \in \mathcal{U} \label{eq:constraint_feasible_set} \\
			&\bm{L}_t,\bm{K}_t = \arg \min_{\bm{L}_t,\bm{K}_t} J(\bm{\theta}) \enspace (\text{Lemma~\ref{lemma:solution_SOC}}) \label{eq:constraint_SOC} \\
			&\E{\bm{x}_t}\,\eqref{eq:SensoSolutionEW}, \bm{\Omega}_t^{\bm{x}}\,\eqref{eq:SensoSolutionCOV} \enspace \text{with} \enspace \bm{L}_t, \bm{K}_t,\bm{\theta} \enspace (\text{Theorem~\ref{theorem:mean_covariance_SOC}}) \label{eq:constraint_mean_covariance}
		\end{align}
	\end{subequations}
	where 
	\begin{align}
		m^{\text{VAF}}_i &= \left(1 - \frac{\sum_{t=0}^{N} \left(\left(\E{\bm{M}\bm{x}_t}\right)_i - \hat{m}_{i,t}\right)^2}{\sum_{t=0}^{N} \left(\hat{m}_{i,t} - \frac{1}{N+1}\sum_{t}\hat{m}_{i,t}\right)^2}\right), \label{eq:m_VAF} \\
		\Omega^{\text{VAF}}_{ij} &= \left(1 - \frac{\sum_{t=0}^{N} \left(\left(\bm{M}\bm{\Omega}^{\bm{x}}_{t}\bm{M}^{\intercal}\right)_{ij} - \hat{\Omega}_{ij,t}^{\bm{x}^*}\right)^2}{\sum_{t=0}^{N} \left(\hat{\Omega}_{ij,t}^{\bm{x}^*} - \frac{1}{N+1}\sum_{t}\hat{\Omega}_{ij,t}^{\bm{x}^*}\right)^2}\right) \label{eq:Omega_VAF}	
	\end{align}
	with $i,j \in \{1,\dots,\bar{n}\}$ and $\bm{w}_{\text{m}} \in \mathbb{R}^{\bar{n}}$ as well as $\bm{w}_{\text{v}} \in \mathbb{R}^{\bar{n}\bar{n}}$ are arbitrary weighting vectors\comment{to account, e.g. in real human measurement data, for a better approximation of the mean compared to the covariance values}.
\end{problem}

In general, $J_{\ISOC}(\bm{\theta},\hat{\bm{m}}_t,\hat{\bm{\Omega}}_t^{\bm{x}^*})$ is non-convex (cf.~Assumption~\ref{assumption:J_ISOC_nonconvex}). Furthermore, Assumption~\ref{assumption:J_ISOC_twcontdiff}\footnote{In references on bi-level-based algorithms (see e.g. \cite{Karg.2023a,Oguz.2018,Berret.2011,Mombaur.2010}), the necessity of derivative-free optimization solvers is postulated. However, with our definition of the optimization problem~\eqref{eq:ISOC_parameter_opt}, we show in Subsection~\ref{subsec:numerical_assm_validation} numerically that $J_{\ISOC}(\bm{\theta},\hat{\bm{m}}_t,\hat{\bm{\Omega}}_t^{\bm{x}^*})$ has indeed sufficient differentiability characteristics.} motivates the use of derivative-based optimization solvers with numerically approximated derivatives since analytical expressions of the derivatives cannot be given due to the bi-level structure (see \eqref{eq:constraint_SOC} and \eqref{eq:constraint_mean_covariance}). Finally, Lemma~\ref{lemma:global_solution_set} characterizes the set~$\mathcal{T}$ of global optimizers of \eqref{eq:ISOC_parameter_opt}.
\begin{assumption} \label{assumption:J_ISOC_nonconvex}
	$J_{\ISOC}(\bm{\theta},\hat{\bm{m}}_t,\hat{\bm{\Omega}}_t^{\bm{x}^*})$ is non-convex w.r.t. $\bm{\theta}$.
\end{assumption}

\begin{assumption} \label{assumption:J_ISOC_twcontdiff}
	$J_{\ISOC}(\bm{\theta},\hat{\bm{m}}_t,\hat{\bm{\Omega}}_t^{\bm{x}^*})$ is (at least) twice continuously differentiable w.r.t. $\bm{\theta}$.
\end{assumption}

\begin{definition} \label{definition:global_solutions}
	Let $\mathcal{T} \subseteq \mathcal{U}$ be the set of global optimizers of \eqref{eq:ISOC_parameter_opt}.
\end{definition}

\begin{lemma} \label{lemma:global_solution_set}
	Let Assumption~\ref{assumption:J_ISOC_twcontdiff} hold. Then, $\mathcal{T} \neq \emptyset$. Moreover, if Assumption~\ref{assumption:GT_data} holds additionally, every $\bm{\theta} \in \mathcal{T}$ solves Problem~\ref{problem:ISOC} and for every $\bm{\theta} \in \mathcal{T}$ $J_{\ISOC}(\bm{\theta},\hat{\bm{m}}_t,\hat{\bm{\Omega}}_t^{\bm{x}^*})=-1$. 
\end{lemma}
\begin{proof}
	Since $\mathcal{U}$ is compact and $J_{\ISOC}(\bm{\theta},\hat{\bm{m}}_t,\hat{\bm{\Omega}}_t^{\bm{x}^*})$ continuous w.r.t. $\bm{\theta}$ (cf.~Assumption~\ref{assumption:J_ISOC_twcontdiff}), (at least) one global maximizer and minimizer exists for $J_{\ISOC}(\bm{\theta},\hat{\bm{m}}_t,\hat{\bm{\Omega}}_t^{\bm{x}^*})$ on $\mathcal{U}$. Hence, $\mathcal{T} \neq \emptyset$ follows. 
	Now, let Assumption~\ref{assumption:GT_data} hold additionally. First, we have $J_{\ISOC}(\bm{\theta}^*,\hat{\bm{m}}_t,\hat{\bm{\Omega}}_t^{\bm{x}^*})=-1$, which is the global minimal value due to the used VAF metric for $m_i^{\text{VAF}}$~\eqref{eq:m_VAF} and $\Omega_{ij}^{\text{VAF}}$~\eqref{eq:Omega_VAF}, and since $\bm{\theta}^* \in \mathcal{U}$ (see Definition~\ref{definition:feasible_set}), $\bm{\theta}^* \in \mathcal{T} \neq \emptyset$ follows and every $\bm{\theta} \in \mathcal{T}$ needs to yield $J_{\ISOC}(\bm{\theta},\hat{\bm{m}}_t,\hat{\bm{\Omega}}_t^{\bm{x}^*})=-1$.
\end{proof}

\begin{remark} \label{remark:existence_solution_set}
	If Assumption~\ref{assumption:GT_data} does not hold, $\mathcal{T} \neq \emptyset$ still holds according to Lemma~\ref{lemma:global_solution_set}. However, $J_{\ISOC}(\bm{\theta},\hat{\bm{m}}_t,\hat{\bm{\Omega}}_t^{\bm{x}^*})=-1$ cannot be guaranteed for $\bm{\theta} \in \mathcal{T}$. The elements of $\mathcal{T}$ yield the best possible fit of the LQS model to $\hat{\bm{m}}_t$ and $\hat{\bm{\Omega}}_t^{\bm{x}^*}$ computed with the GT data. Assumption~\ref{assumption:GT_data} is not fulfilled if either too few realizations are measured for sufficient approximations of mean and covariance values or the LQS model is not valid to describe the GT data.
\end{remark}

\subsection{Upper Level Optimization: Threshold Random Linkage Algorithm with Approximation of Regions of Attraction} \label{subsec:upper_level_opt_TRLwARoA}

In this subsection, we propose two algorithms, together with their convergence proofs, to solve Problem~\ref{problem:ISOC_parameter_opt}. Lemma~\ref{lemma:global_solution_set} shows that under Assumption~\ref{assumption:GT_data} these solutions solve Problem~\ref{problem:ISOC} as well. Algorithm~\ref{algorithm:TRLwARoA} is our main algorithm, influenced by \cite{Locatelli.1999,RinnooyKan.1987b} and called Threshold Random Linkage Algorithm with Approximation of Regions of Attraction (TRLwARoA). The main idea of our algorithms is to use a derivative-based constrained local optimization solver~$\Upsilon$ according to Definition~\ref{definition:localSolverRoA} that converges to a local minimizer of \eqref{eq:ISOC_parameter_opt}. Then, this local solver is combined with a globalization strategy to account for the non-convexity of $J_{\ISOC}(\bm{\theta})$.
\begin{definition} \label{definition:localSolverRoA}
	Let Assumption~\ref{assumption:J_ISOC_twcontdiff} hold and let a feasible set~$\mathcal{U}$ (cf.~Definition~\ref{definition:feasible_set}) be given. A constrained local optimization solver $\Upsilon: \mathcal{U} \rightarrow \mathcal{U}$ yields $\bm{\theta}_{\min} = \Upsilon(\bm{\theta}_0)$, where $\bm{\theta}_{\min} \in \Theta_{\min}$ and $\bm{\theta}_0 \in \Theta_{\RoA}(\Theta_{\min})$. Here, $\Theta_{\min}$ denotes a connected subset of all local minimizers $\{ \bm{\theta}_{\min} \in \mathcal{U}: \exists \mathcal{N}(\bm{\theta}_{\min}) \text{ such that } J_{\ISOC}(\bm{\theta}) \geq J_{\ISOC}(\bm{\theta}_{\min}), \forall \bm{\theta} \in \mathcal{N}(\bm{\theta}_{\min}) \}$ (with $\mathcal{N}(\bm{\theta}_{\min})$ as neighborhood of $\bm{\theta}_{\min}$) and $\Theta_{\RoA}(\Theta_{\min})$ the Region of Attraction (RoA) of $\Theta_{\min}$, i.e. $\Theta_{\RoA}(\Theta_{\min}) = \{\bm{\theta} \in \mathcal{U}: \exists \bm{\theta}_{\min} \in \Theta_{\min} \text{ such that } \forall \bm{\theta} \in \mathcal{N}(\bm{\theta}_{\min}) \, \exists \bm{\theta}_{\min}' \in \Theta_{\min}: \bm{\theta}_{\min}' = \Upsilon(\bm{\theta}) \}$.
\end{definition}
\comment{note: definition does not say anything about the non-emptiness of the RoA - the RoA of connected sets of local minima can be empty in general, important is only the non-emptiness of the RoA of $\mathcal{T}$, so from at least one connected subset of $\mathcal{T}$ - on the other side, the implementation of the local solver should always converge to a local minimum (if convergence occurred), should normally be the case; however, if convergence to only KKT point, i.e. e.g. saddle point, occurs, the algorithm still converges if sample point from non-empty RoA of $\mathcal{T}$ is chosen and $\gamma$, $v$ are sufficiently small ($v$ maintains here also approximation of RoA of that KKT point) - through the first-order optimality condition stopping criterion, in our case also the convergence based on small changes in the objective function is guaranteed, solver does not move inside connected set of local minima}

In Definition~\ref{definition:localSolverRoA}, we consider connected sets of local minima due to the well-known ambiguity of solutions of ISOC problems \cite{Karg.2023a,Molloy.2022}. For example, if Assumption~\ref{assumption:GT_data} is fulfilled, every $\bm{\theta}=\mat{ \lambda {\bm{s}^*}^\intercal & {\bm{\sigma}^*}^\intercal }^\intercal \in \mathcal{T}$ with $\lambda \in \mathbb{R}_{\geq 0}$ and such that $\bm{\theta} \in \mathcal{U}$. Hence, sufficient conditions for strict optimizers are normally not fulfilled in ISOC problems. The shape of the RoA~$\Theta_{\RoA}$ of a set $\Theta_{\min}$ depends on the concrete implementation of the used local solver. However, there are convergence results for basic implementations of SQP and interior-point (IP) algorithms (see e.g. \cite[Theorem~18.3, Theorem~19.1]{Nocedal.2006}) that are applicable for \eqref{eq:ISOC_parameter_opt}. Later we use an IP algorithm as local solver. Since $J_{\ISOC}(\bm{\theta})$ as well as the inequality constraints given by $\mathcal{U}$ are continuously differentiable and since the LICQ hold $\forall \bm{\theta} \in \mathcal{U}$ (cf.~Lemma~\ref{lemma:LICQ}), from \cite[Theorem~19.1]{Nocedal.2006} the convergence of IP algorithms to first-order optimal points follows. 
\begin{lemma} \label{lemma:LICQ}
	Let a feasible set~$\mathcal{U}$ according to Definition~\ref{definition:feasible_set} be given. The linear independence constraint qualification (LICQ) holds $\forall \bm{\theta} \in \mathcal{U}$.
\end{lemma}
\begin{proof}
	Since $\mathcal{U}$ is a hyperrectangle, all constraints are given by $\theta_i - a_i \geq 0$ and $-\theta_i + b_i \geq 0$ ($i \in \{1,\dots,\Theta\}$). For the set of their gradients $\{ \bm{e}_i, -\bm{e}_i, i \in \{1,\dots,\Theta\} \}$ ($\bm{e}_i$ as $i$-th standard basis vector of $\mathbb{R}^{\Theta}$) follows. Now, since $b_i > a_i$, the linear dependent vectors $\bm{e}_i$ and $-\bm{e}_i$ cannot correspond to active constraints at the same point. Thus, from the independence of $\bm{e}_i$ and $\bm{e}_j$ ($\forall i,j \in \{1,\dots,\Theta\}$ with $i\neq j$), we conclude the LICQ $\forall \bm{\theta} \in \mathcal{U}$.
\end{proof}

Suppose a constrained local optimization solver~$\Upsilon$ according to Definition~\ref{definition:localSolverRoA} is given, we can find a global optimizer $\bm{\theta}^{\bigtriangleup}$ of \eqref{eq:ISOC_parameter_opt} by starting it from $k_{\max}$ uniformly sampled points in $\mathcal{U}$. This leads to Algorithm~\ref{algorithm:PureMS} and Theorem~\ref{theorem:conv_pure_multistart}.
\begin{algorithm}[t] 
	\caption{Pure Multi-Start ISOC Algorithm.}
	\label{algorithm:PureMS}
	\DontPrintSemicolon
	\KwIn{$\hat{\bm{m}}_t$, $\hat{\bm{\Omega}}_t^{\bm{x}^*}$, $\mathcal{U}$, $k_{\max}$}
	\KwOut{$\bm{\theta}^{\bigtriangleup}$, $J_{\ISOC}(\bm{\theta}^{\bigtriangleup},\hat{\bm{m}}_t,\hat{\bm{\Omega}}_t^{\bm{x}^*})$}
	
	Draw $\bm{\theta}^{(1)},\dots,\bm{\theta}^{(k_{\max})}$ samples from uniform distribution in $\mathcal{U}$ \;
	Set $k=1$ and $J_{\OPT} = \infty$ \;
	\While{$k\leq k_{\max}$}{
		Compute $\bm{\theta}_{\min} = \Upsilon(\bm{\theta}^{(k)})$ with $\Upsilon$ applied to \eqref{eq:ISOC_parameter_opt} \;
		\If{$J_{\ISOC}(\bm{\theta}_{\min},\hat{\bm{m}}_t,\hat{\bm{\Omega}}_t^{\bm{x}^*})<J_{\OPT}$}{
			$\bm{\theta}^{\bigtriangleup} = \bm{\theta}_{\min}$ \;
			$J_{\text{OPT}} = J_{\ISOC}(\bm{\theta}_{\min},\hat{\bm{m}}_t,\hat{\bm{\Omega}}_t^{\bm{x}^*})$ \;
		}
		$k \leftarrow k+1$ \;
	}
	\Return{$\bm{\theta}^{\bigtriangleup}$, $J_{\ISOC}(\bm{\theta}^{\bigtriangleup},\hat{\bm{m}}_t,\hat{\bm{\Omega}}_t^{\bm{x}^*})$}
\end{algorithm}

\begin{theorem} \label{theorem:conv_pure_multistart}
	Let Assumption~\ref{assumption:J_ISOC_twcontdiff} hold and let a feasible set~$\mathcal{U}$ (cf.~Definition~\ref{definition:feasible_set}) as well as a constrained local optimization solver~$\Upsilon$ according to Definition~\ref{definition:localSolverRoA} be given. Algorithm~\ref{algorithm:PureMS} yields $\bm{\theta}^{\bigtriangleup} \in \mathcal{T}$ (cf.~Definition~\ref{definition:global_solutions}) if (at least) one sample point $\bm{\theta}^{(k)} \in \Theta_{\RoA}(\mathcal{T})$\comment{$\mathcal{T}$ could be the union of several connected sets of global minima}.
\end{theorem}
\begin{proof}
	From Definition~\ref{definition:localSolverRoA}, we have $\bm{\theta}^{\bigtriangleup} = \Upsilon(\bm{\theta}^{(k)})$ with $\bm{\theta}^{(k)} \in  \Theta_{\RoA}(\mathcal{T})$.
\end{proof}

The probability of choosing a uniformly sampled point $\bm{\theta}^{(k)} \in \Theta_{\RoA}(\mathcal{T})$ is given by \cite{RinnooyKan.1987a}
\begin{align} \label{eq:probabilityUniSampling}
	P(\bm{\theta}^{(k)} \! \! \in \! \Theta_{\RoA}(\mathcal{T})) = 1 \! -\! \left( \! 1 \! - \! \frac{\vol{\Theta_{\RoA}(\mathcal{T})}}{\vol{\mathcal{U}}} \right)^{k_{\max}} \!,
\end{align}
where $\vol{\cdot}$ denotes the volume of the corresponding space. Hence, since $\mathcal{T} \neq \emptyset$ follows from Lemma~\ref{lemma:global_solution_set}, an appropriate local solver needs to guarantee $\vol{\Theta_{\RoA}(\mathcal{T})} \neq 0$. Then, $k_{\max} \rightarrow \infty$ guarantees $P(\bm{\theta}^{(k)} \in \Theta_{\RoA}(\mathcal{T})) \rightarrow 1$. 
Starting a local solver from every sample point $\bm{\theta}^{(k)}$ results in impractical computation times. Therefore, we include two filters to determine if a local solver is started at the current sample point $\bm{\theta}^{(k)}$. First, a distance filter is implemented motivated by \cite{Locatelli.1999}. A local solver is only started from $\bm{\theta}^{(k)}$ if the distance
\begin{align} \label{eq:delta}
	\delta^{(k)} = \begin{cases}
		\delta^{(k)}_c & \text{if } \exists i \! < \! k \!:\! J_{\ISOC}(\bm{\theta}^{(i)}) \! < \! J_{\ISOC}(\bm{\theta}^{(k)}) \\
		\infty & \text{otherwise}
	\end{cases},
\end{align}
where $\delta^{(k)}_c = \min_{i<k} \{ \norm{\bm{\theta}^{(k)}-\bm{\theta}^{(i)}}_2:  J_{\ISOC}(\bm{\theta}^{(i)}) <$ $J_{\ISOC}(\bm{\theta}^{(k)}) \}$,
to previously checked sample points $\bm{\theta}^{(i)}$ ($i\in\{1,\dots,k-1\}$) is greater than a threshold $\alpha$. Sample points with $J_{\ISOC}(\bm{\theta}^{(k)})\leq J_{\ISOC}(\bm{\theta}^{(i)})$ ($\forall i\in\{1,\dots,k-1\}$) are always used as starting points ($\delta^{(k)} = \infty$). With the distance filter, a sample point density in $\mathcal{U}$ independent from the size of $\mathcal{U}$ is maintained by adapting $\alpha$ to the size of $\mathcal{U}$. Suppose $\mathcal{U}$ is divided into hyperspheres~$\mathcal{S}$ with radius $\alpha$. Then, $\alpha$ is chosen such that the number of hyperspheres fitting into $\mathcal{U}$ is constant: $\const = \frac{\vol{\mathcal{U}}}{\vol{\mathcal{S}}}$. This leads to
\begin{align}
	\const &= \frac{\Gamma\left( \frac{\Theta}{2}+1 \right)\vol{\mathcal{U}}}{\pi^{\frac{\Theta}{2}}\alpha^{\Theta}} \nonumber \\
	\Rightarrow \alpha &= \gamma \left( \Gamma\left( \frac{\Theta}{2}+1 \right) \prod_{i=1}^{\Theta} (b_i - a_i) \right)^{\frac{1}{\Theta}}, \label{eq:alpha}
\end{align}
where $\gamma = \pi^{-\frac{1}{2}} \const^{-\frac{1}{\Theta}}$ is a tuning parameter for the distance filter and $\Gamma$ denotes the gamma function.

The second filter to further reduce the number of local solver starts is a RoA filter. If a local minimum is found, i.e. $\bm{\theta}_{\min}^{(l)} = \Upsilon(\bm{\theta}^{(k)})$, the RoA~$\Theta_{\RoA}(\bm{\theta}_{\min}^{(l)})$ is approximated by a hypersphere~$\{ \bm{\theta} \in \mathcal{U}: \norm{\bm{\theta}-\bm{\theta}_{\min}^{(l)}}_2 < v \delta_{\RoA}^{(l)}\}$, where $\delta_{\RoA}^{(l)} = \norm{\bm{\theta}^{(k)}-\bm{\theta}_{\min}^{(l)}}_2$ and $v$ denotes a tuning parameter guaranteeing Assumption~\ref{assumption:RoA_approximation}.

Overall, our TRLwARoA is given by Algorithm~\ref{algorithm:TRLwARoA} and in Theorem~\ref{theorem:conv_TRLwARoA}, its global convergence is shown.
\begin{algorithm}[t] 
	\caption{TRLwARoA ISOC Algorithm.}
	\label{algorithm:TRLwARoA}
	\DontPrintSemicolon
	\KwIn{$\hat{\bm{m}}_t$, $\hat{\bm{\Omega}}_t^{\bm{x}^*}$, $\mathcal{U}$, $k_{\max}$, $\gamma$, $v$}
	\KwOut{$\bm{\theta}^{\bigtriangleup}$, $J_{\ISOC}(\bm{\theta}^{\bigtriangleup},\hat{\bm{m}}_t,\hat{\bm{\Omega}}_t^{\bm{x}^*})$}
	
	Draw $\bm{\theta}^{(1)},\dots,\bm{\theta}^{(k_{\max})}$ samples from uniform distribution in $\mathcal{U}$ \;
	Set $k=1$, $l=0$ and $J_{\OPT} = \infty$ \;
	Compute $\alpha$~\eqref{eq:alpha} \;
	\While{$k\leq k_{\max}$}{
		Compute $\delta^{(k)}$~\eqref{eq:delta} \;
		\If{$\delta^{(k)} > \alpha \wedge \norm{\bm{\theta}^{(k)}-\bm{\theta}_{\min}^{(i)}}_2 \geq v\delta_{\RoA}^{(i)}, \forall i \in \{1,\dots,l\}$ }{
		$l \leftarrow l+1$ \;
		Compute $\bm{\theta}_{\min}^{(l)} = \Upsilon(\bm{\theta}^{(k)})$ with $\Upsilon$ to \eqref{eq:ISOC_parameter_opt} \;
		$\delta_{\RoA}^{(l)} = \norm{\bm{\theta}^{(l)}_{\min}-\bm{\theta}^{(k)}}_2$ \;
		\If{$J_{\ISOC}(\bm{\theta}_{\min}^{(l)},\hat{\bm{m}}_t,\hat{\bm{\Omega}}_t^{\bm{x}^*})<J_{\OPT}$}{
			$\bm{\theta}^{\bigtriangleup} = \bm{\theta}_{\min}^{(l)}$ \;
			$J_{\text{OPT}} = J_{\ISOC}(\bm{\theta}_{\min}^{(l)},\hat{\bm{m}}_t,\hat{\bm{\Omega}}_t^{\bm{x}^*})$ \;
		} 
		}
		$k \leftarrow k+1$ \;
	}
	\Return{$\bm{\theta}^{\bigtriangleup}$, $J_{\ISOC}(\bm{\theta}^{\bigtriangleup},\hat{\bm{m}}_t,\hat{\bm{\Omega}}_t^{\bm{x}^*})$}
\end{algorithm}

\begin{definition} \label{definition:optimal_level_set}
	Let $\mathcal{M}^{J_{\ISOC}}(m) = \{\bm{\theta}\in\mathcal{U}: J_{\ISOC}(\bm{\theta})\leq m \}$ with $m\in [-1,\infty)$ be the $m$-level set of $J_{\ISOC}$. The set $\mathcal{M}^{J_{\ISOC}}(m^*)$ is the optimal level set where $m^*$ is the highest value such that $\mathcal{T}$ are the only local minima in $\mathcal{M}^{J_{\ISOC}}(m^*)$.
\end{definition}

\begin{assumption} \label{assumption:RoA_approximation}
	For all $\bm{\theta}_{\min}^{(l)}$ found by Algorithm~\ref{algorithm:TRLwARoA} $\{ \bm{\theta} \in \mathcal{U}: \norm{\bm{\theta}-\bm{\theta}_{\min}^{(l)}}_2 < v \delta_{\RoA}^{(l)}\} \subseteq \Theta_{\RoA}(\bm{\theta}_{\min}^{(l)})$\comment{it is possible (and allowed) that more than one element of one connected set $\Theta_{\min}$ of local minima is found; the union of their RoA approximating hyperspheres approximates then $\Theta_{\RoA}(\Theta_{\min})$}.
\end{assumption}

\begin{theorem} \label{theorem:conv_TRLwARoA}
	Let Assumptions~\ref{assumption:J_ISOC_twcontdiff} and \ref{assumption:RoA_approximation} hold and let a feasible set~$\mathcal{U}$ (cf.~Definition~\ref{definition:feasible_set}) as well as a constrained local optimization solver~$\Upsilon$ according to Definition~\ref{definition:localSolverRoA} be given. Algorithm~\ref{algorithm:TRLwARoA} yields $\bm{\theta}^{\bigtriangleup} \in \mathcal{T}$ (cf.~Definition~\ref{definition:global_solutions}) if (at least) one sample point $\bm{\theta}^{(k)} \in \Theta_{\RoA}(\mathcal{T}) \cap \mathcal{M}^{J_{\ISOC}}(m^*)$.
\end{theorem}
\begin{proof}
	We need to verify that $\bm{\theta}^{(k)} \in \Theta_{\RoA}(\mathcal{T}) \cap \mathcal{M}^{J_{\ISOC}}(m^*)$ is not filtered by the distance and RoA filter. Then, $\bm{\theta}^{\bigtriangleup} \in \mathcal{T}$ follows from Definition~\ref{definition:localSolverRoA}. First, since $\bm{\theta}^{(k)}$ cannot be in the RoA of different local minima, $\bm{\theta}^{(k)}$ is not filtered by the RoA filter due to Assumption~\ref{assumption:RoA_approximation} and $\bm{\theta}^{(k)} \in \Theta_{\RoA}(\mathcal{T})$. Furthermore, $\bm{\theta}^{(k)}$ is not filtered by the distance filter since $\forall \bm{\theta} \in \mathcal{M}^{J_{\ISOC}}(m^*)$ $J_{\ISOC}(\bm{\theta})\leq J_{\ISOC}(\bm{\theta}^{(i)})$ with $\bm{\theta}^{(i)} \notin \mathcal{M}^{J_{\ISOC}}(m^*)$ and uniformly sampled in $\mathcal{U}$.
\end{proof}

\begin{remark} \label{remark:tuning_parameter_gamma}
	Typically, $\mathcal{M}^{J_{\ISOC}}(m^*) \subseteq \Theta_{\RoA}(\mathcal{T})$\footnote{From Lemma~\ref{lemma:global_solution_set}, $\mathcal{T}\neq \emptyset$ and due to the assumed continuity of $J_{\ISOC}(\bm{\theta})$ (cf.~Assumption~\ref{assumption:J_ISOC_twcontdiff}), we have $\vol{\mathcal{M}^{J_{\ISOC}}(m^*)} \neq 0$}. Hence, by reducing the number of local solver starts with the distance filter, the requirement on the sampling procedure (at least one sample point~$\bm{\theta}^{(k)} \in \Theta_{\RoA}(\mathcal{T}) \cap \mathcal{M}^{J_{\ISOC}}(m^*)$) gets more restrictive. However, by lowering $\gamma$ this more restrictive assumption can be relaxed and for $\gamma \rightarrow 0$ $\bm{\theta}^{(k)} \in \Theta_{\RoA}(\mathcal{T})$ is sufficient (cf.~Theorem~\ref{theorem:conv_pure_multistart}). Thus, in case of problems with small $\vol{\mathcal{M}^{J_{\ISOC}}(m^*)}$ $\gamma$ can be used to maintain convergence under the best possible performance improvement of the distance filter.
\end{remark}

\begin{remark} \label{remark:local_solver}
	Theorem~\ref{theorem:conv_pure_multistart} and \ref{theorem:conv_TRLwARoA} show that the main requirement on the local solver~$\Upsilon$ (cf.~Definition~\ref{definition:localSolverRoA}) is to guarantee $\vol{\Theta_{\RoA}(\mathcal{T})} \neq 0$, which can be motivated by the convergence of e.g. basic IP methods to KKT points of \eqref{eq:ISOC_parameter_opt} (see before). Note that the algorithms do not depend on $\vol{\Theta_{\RoA}(\Theta_{\min})} \neq 0$ for every $\Theta_{\min}$. Furthermore, most implementations of constrained local optimization solvers will even guarantee convergence to local minima, e.g. due to corrections on Hessian approximations, which is beneficial for Algorithm~\ref{algorithm:TRLwARoA} due to the RoA filter.
\end{remark}

\subsection{Lower Level Optimization} \label{subsec:lower_level_opt}

The lower level of our bi-level-based ISOC algorithm is given by the constraints \eqref{eq:constraint_SOC} and \eqref{eq:constraint_mean_covariance} of \eqref{eq:ISOC_parameter_opt} and for their computation Assumption~\ref{assumption:model_structure} is needed. 

In general, the assumption on a linear non-adaptive filter in Lemma~\ref{lemma:solution_SOC} leads to suboptimal solutions to SOC Problem~\ref{problem:SOC}. However, exact solutions are still an open question. They exist for the fully observable case\footnote{The setup of \eqref{eq:ISOC_parameter_opt} for the fully observable case is straightforward and all our algorithms and statements can be applied as well.} ($\bm{y}_t = \bm{x}_t$) \cite{Todorov.2005,Beghi.1998} but in case of partial state information only suboptimal solutions were proposed (see e.g. \cite{Liang.2023,Todorov.2005,Moore.1999,Joshi.1976}). 
\comment{suboptimal solutions in the lower level of bi-level-based I(S)OC algorithms can lead to "false" friends (match with GT data of the suboptimal strategy although model does not fit; however, the trajectories derived by the suboptimal strategy suggest a match) or a rejection of the model hypothesis (although the model matches, but the suboptimal strategy hinders a fit)}


\section{Numerical Example} \label{sec:numerical_example}
After introducing an example system and the implementation of Algorithms~\ref{algorithm:PureMS} and \ref{algorithm:TRLwARoA} in Subsection~\ref{subsec:example_system}, in Subsection~\ref{subsec:numerical_assm_validation}, Assumptions~\ref{assumption:J_ISOC_nonconvex} and \ref{assumption:J_ISOC_twcontdiff} are validated numerically. Furthermore, Lemma~\ref{lemma:solution_SOC} is compared to other (suboptimal) control strategies for Problem~\ref{problem:SOC} and we show that $\vol{\Omega_{\RoA}(\mathcal{T})} \neq 0$ for our used IP implementation. Finally, in Subsection~\ref{subsec:tuning_para_analysis} and \ref{subsec:SoA_comparison} the new TRLwARoA ISOC algorithm is analyzed regarding its tuning parameters and compared to State-of-the-Art (SoA) algorithms.

\subsection{Example System and Implementation} \label{subsec:example_system}

Our simulation example is given by a planar point-to-point human hand movement, where the human hand is modeled as point mass. The exact definition can be found in \cite{Karg.2023a}. The GT parameters~$\bm{\theta}^*$ are given by: $\bm{\theta}^* = \mat{ {\bm{s}^*}^\intercal & {\bm{\sigma}^*}^\intercal }^\intercal$ with $\bm{s}^* = \mat{ 1 & 1 & 0.04 & 0.04 & 0.0004 & 0.0004 & \frac{1}{42} 10^{-5} & \frac{1}{42} 10^{-5} }^\intercal$ and $\bm{\sigma}^* = \mat{ \bm{0}_{1 \times 8}\!\!\!\! & \! 0.02 & \!0.02 & 0.2 & 0.2 & 1 & 1 & 0.5 & 0.1 }^\intercal$. From the GT parameters, $\hat{\bm{m}}_t$ and $\hat{\bm{\Omega}}_t^{\bm{x}^*}$ follow according to Lemma~\ref{lemma:solution_SOC} and Theorem~\ref{theorem:mean_covariance_SOC}. Hence, due to the simulation scenario Assumptions~\ref{assumption:GT_data} and \ref{assumption:model_structure} are fulfilled. Lastly, the lower bounds of the feasible set~$\mathcal{U}$ are defined by $a_i = 0$ ($\forall i \in \{1,\dots,6,9,\dots,16,23,24\}$) and $a_i = 10^{-10}$ ($\forall i \in \{7,8,17,\dots,22\}$) and $\bm{w}_{\text{m}} = \mat{ 0.9 & 0.9 & 0.9 & 0.9 }^\intercal$ and $\bm{w}_{\text{v}} = \mat{ 0.1 & \bm{0}_{1\times 4} & 0.1 & \bm{0}_{1\times 4} & 0.1 & \bm{0}_{1\times 4} & 0.1 }^\intercal$ are used to set up \eqref{eq:ISOC_parameter_opt}. 

The implementation of the algorithms was done in Matlab 2021b on a standard PC with a Ryzen 9 5950X. As local optimization solver~$\Upsilon$, the IP method of the \texttt{fmincon} solver of the Matlab environment is used. First and second order derivatives are numerically approximated. Both ISOC algorithms were implemented via parallel computing on the $16$ cores available on the CPU (scalable to more cores) by starting local optimizations on these cores in parallel.

\subsection{Numerical Validation of the Assumptions} \label{subsec:numerical_assm_validation}

First, we motivate the validity of Assumption~\ref{assumption:J_ISOC_twcontdiff}. Hereto, we compute the $276$ projections of $J_{\ISOC}(\bm{\theta})$ on the $\theta_i$-$\theta_j$-planes ($i,j \in \{1,\dots,24\}$, $j>i$). The value of $J_{\ISOC}(\bm{\theta})$ is calculated $\forall \bm{\theta} = \mat{ \dots \!\!\!\! & \theta_{i-1}^* & \!\!\! \theta_i \!\!\! & \theta_{i+1}^* & \!\!\!\! \dots \!\!\!\! & \theta_{j-1}^* & \!\!\! \theta_j \!\!\! & \theta_{j+1}^* & \!\!\!\! \dots }^\intercal$ with $\theta_{i/j} = \lambda \theta_{i/j}^*$ and $\lambda$ chosen in $101$ equidistant steps from $[0.5,1.5]$ if $\theta_{i/j}^* \neq 0$. If $\theta_{i/j}^* = 0$, $101$ equidistant steps for $\theta_{i/j}$ from $[0,1]$ are chosen. Finally, an evaluation of $J_{\ISOC}(\bm{\theta})$ at equidistant grid points for every $\theta_i$-$\theta_j$-plane follows (see red crosses in Fig.~\ref{fig:theta1theta2JISOC} for $\theta_1$-$\theta_2$-plane). Now, by fitting 2D polynomials of order five to the grid points of $J_{\ISOC}(\bm{\theta})$ for every $\theta_i$-$\theta_j$-plane, we achieve $R^2$ values of $>0.999$ ($R^2=0.99986$ for $\theta_1$-$\theta_2$-plane in Fig.~\ref{fig:theta1theta2JISOC}, minimal value of all polynomial fits). The quantitative and qualitative results (see surface plot in Fig.~\ref{fig:theta1theta2JISOC}) of this fitting procedure and the approximated shape of $J_{\ISOC}(\bm{\theta})$ through the evaluation at the grid points (the results for the other $\theta_i$-$\theta_j$-planes look very similar) provide strong indications that $J_{\ISOC}(\bm{\theta})$ is an analytical function like the polynomial fits. Thus, $J_{\ISOC}(\bm{\theta})$ can be assumed to be (at least) twice continuously differentiable (cf.~Assumption~\ref{assumption:J_ISOC_twcontdiff}). Fig.~\ref{fig:nonConvexJISOC} depicts the 1D projection of $J_{\ISOC}(\bm{\theta})$ at $\bm{\theta} = \bm{\theta}^* + \lambda (\bm{\theta}_{\min}^{(l)} - \bm{\theta}^*)$, where $\bm{\theta}_{\min}^{(l)}$ denotes a local minimum. It clearly shows the assumed non-convexity of $J_{\ISOC}(\bm{\theta})$ (cf.~Assumption~\ref{assumption:J_ISOC_nonconvex}).
\begin{figure}[t]
	\centering
	\begin{subfigure}[t]{1.5in}
		\includegraphics[width=1.5in]{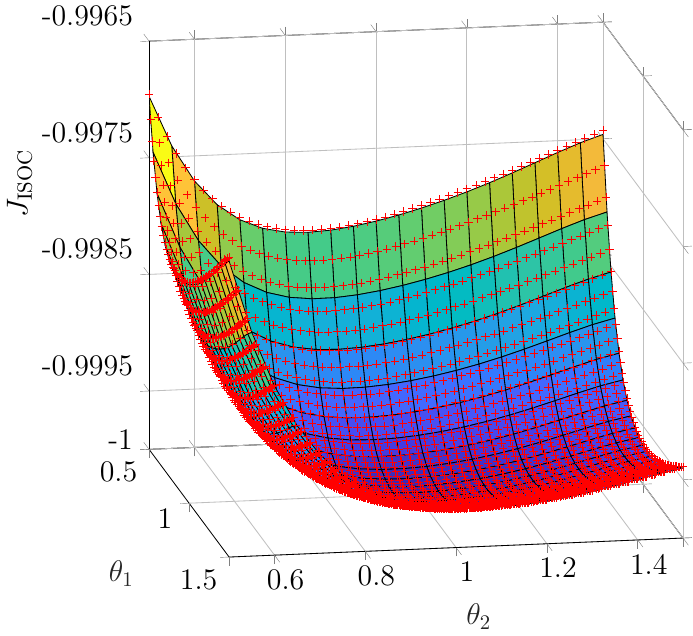}
		\caption{Evaluation of $J_{\ISOC}(\bm{\theta})$ at equidistant grid points in $\theta_1$-$\theta_2$-plane together with surface of polynomial fit.}
		\label{fig:theta1theta2JISOC}
	\end{subfigure}
	\begin{subfigure}[t]{1.5in}
		\includegraphics[width=1.5in]{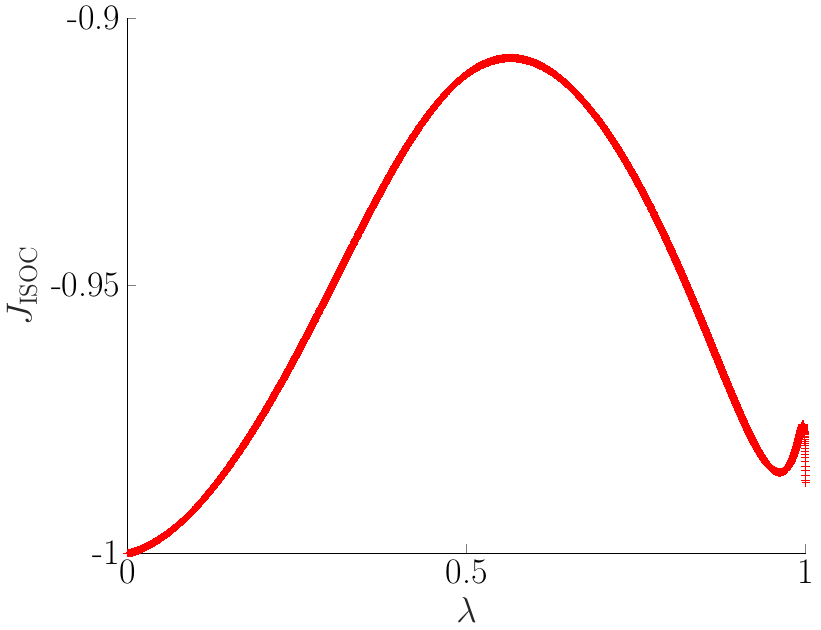}
		\caption{Evaluation of $J_{\ISOC}(\bm{\theta})$ at $\bm{\theta} = \bm{\theta}^* + \lambda (\bm{\theta}_{\min}^{(l)} - \bm{\theta}^*)$.}
		\label{fig:nonConvexJISOC}
	\end{subfigure}
	\caption{Validation of Assumptions~\ref{assumption:J_ISOC_nonconvex} and \ref{assumption:J_ISOC_twcontdiff}.}
	\label{fig:assumptionValidation}
\end{figure}

Next, we look at $\Theta_{\RoA}(\mathcal{T})$. By applying Algorithm~\ref{algorithm:PureMS} with $k_{\max} = 20000$ and $b_i=2$ ($\forall i \in \{1,\dots,24\}$) to the GT data, $692$ local solver runs converge to a $\bm{\theta}^{\bigtriangleup} \in \mathcal{T}$. Hence, $\vol{\Theta_{\RoA}(\mathcal{T}} \neq 0$ follows for the used IP implementation and furthermore, we approximate $\frac{\vol{\Theta_{\RoA}(\mathcal{T})}}{\vol{\mathcal{U}}} \approx \frac{692}{20000}=0.0346$ and have with $k_{\max}=10000$ $P(\bm{\theta}^{(k)} \in \Theta_{\RoA}(\mathcal{T})) \approx 1$ according to \eqref{eq:probabilityUniSampling}.

Finally, following the discussion on the suboptimality of the solution to the lower level optimization in Subsection~\ref{subsec:lower_level_opt}, Table~\ref{table:controlStrategyComparison} shows the comparison of two other suboptimal control strategies to the one of Lemma~\ref{lemma:solution_SOC}. The comparison is performed for $20$ parameters~$\bm{\theta}$: $\bm{\theta}^*$ and $\bm{\theta}$ with $\theta_i = \lambda \theta_i^*$ ($\lambda$ randomly from $[0.5,1.5]$) if $\theta_i^*\neq 0$, $\theta_i$ randomly from $[0,1]$ otherwise ($\forall i \in \{1,\dots,24\}$)\comment{for thesis: more parameters, i.e. around 100}. Both strategies can lead to an improvement as well as a worsening of $J$ with negligible relative change in average. Hence, Lemma~\ref{lemma:solution_SOC} is the best choice since it furthermore enables the use of Theorem~\ref{theorem:mean_covariance_SOC} due to its non-adaptiveness. 
\begin{table}[t]
	\centering
	\caption{Mean, maximal and minimal relative change~$\triangle J$ of performance criterion \eqref{eq:cost_function} with different control strategies. Strategy of Lemma~\ref{lemma:solution_SOC} (non-adaptive $\bm{K}_t$) serves as the reference value.}
	\label{table:controlStrategyComparison}
	\renewcommand{\arraystretch}{1.15}
	\begin{tabular}{|c|C{2.5cm}|C{2.5cm}|}
		\hline
		& Lemma~\ref{lemma:solution_SOC} with adaptive $\bm{K}_t$ \cite{Todorov.2005} vs. non-adaptive $\bm{K}_t$ & \cite{Moore.1999} vs. Lemma~\ref{lemma:solution_SOC} with non-adaptive $\bm{K}_t$ \\
		\hline
		$\triangle J^{\mean}$ & $-1.1\,\%$ & $1.4\,\%$ \\
		\hline
		$\triangle J^{\max}$ & $2.9\,\%$ & $5.7\,\%$ \\
		\hline
		$\triangle J^{\min}$ & $-7.0\,\%$ & $-3.8\,\%$  \\
		\hline
	\end{tabular}
\end{table}

\subsection{Analysis of the Tuning Parameters of Algorithm~\ref{algorithm:TRLwARoA}} \label{subsec:tuning_para_analysis}

In the following, we analyze the convergence behavior and the computation time of our new TRLwARoA ISOC Algorithm~\ref{algorithm:TRLwARoA} w.r.t. to different choices of its tuning parameters, i.e. upper bounds~$b_i$ of feasible set $\mathcal{U}$, tuning parameters~$\gamma$ of the distance filter and $v$ of the RoA filter. We set $k_{\max} = 10000$ since it is sufficient for $P(\bm{\theta}^{(k)} \in \Theta_{\RoA}(\mathcal{T})) \approx 1$ (see Subsection~\ref{subsec:numerical_assm_validation}). Table~\ref{table:tuningParameterRL} shows the parameter combinations evaluated, where $b_i$ is chosen equally $\forall i \in \{1,\dots,24\}$. With every tuning parameter combination, Algorithm~\ref{algorithm:TRLwARoA} was executed $10$ times. As evaluation metrics, the average computation time~$t_{\comp}^{\mean}$, the average number of started local solvers~$\# \Upsilon^{\mean}$ and the worst objective function value~$J_{\ISOC}^{\max}$ achieved during these $10$ runs are used. Furthermore, we define a run of Algorithm~\ref{algorithm:TRLwARoA} as converged to a global optimizer if for the achieved $J_{\ISOC}$ value $J_{\ISOC}\leq -0.999$ holds. The number of converged runs is denoted as $\# l$. Table~\ref{table:tuningParameterRL} shows $\# \Upsilon^{\mean}$, $\# l$ and $t_{\comp}^{\mean}$ for the evaluated tuning parameter combinations. With $v=0$ Assumption~\ref{assumption:RoA_approximation} is guaranteed and for $\gamma=0.6$ all $10$ runs converge. By increasing $v$ to $v=0.7$, Assumption~\ref{assumption:RoA_approximation} can still be considered as fulfilled ($\# l = 10$) and due to the RoA filter, performance is improved significantly: $\# \Upsilon^{\mean}$ and $t_{\comp}^{\mean}$ are nearly halved\footnote{Mean number of local solver starts~$\# \Upsilon^{\mean}$ is nearly independent from the choice of bounds due to the adaptive threshold~$\alpha$~\eqref{eq:alpha}.}\comment{higher computation times in case of $b_i = 20$ can result from higher computation times for local solvers, more complex local optimizations}. Now, by increasing $\gamma$ to $\gamma = 0.7$, $\# \Upsilon^{\mean}$ and $t_{\comp}^{\mean}$ are reduced by more than $80\,\%$ and $70\,\%$, respectively (for $v=0$ and $v=0.7$). However, with $\gamma=0.7$ convergence does only occur in $7-9$ of the $10$ runs, but a local minimum with an objective function value close to the global one is always achieved in the not-converged runs: $J_{\ISOC}^{\max}<-0.98$\comment{these nearly global minima lead to a very small $\vol{\mathcal{M}^{J_{\ISOC}}(m^*)}$, but indicate also that parameters exist that are nearly globally optimal and achieve w.r.t. to the measured states reasonable predictions; most likely such near global minima disappear when more states are measured; one can also define such near global solutions and then, $\vol{\mathcal{M}^{J_{\ISOC}}(m^*)}$ gets bigger and the requirement on the uniform sampling less restrictive}. Due to these near global minima, $\vol{\mathcal{M}^{J_{\ISOC}}(m^*)}$ is small in our case and we need to choose $\gamma$ according to Remark~\ref{remark:tuning_parameter_gamma}, sufficiently small but as large as possible. For $\gamma=0.8$ or $v=0.9$, $\# l \leq 5$ except for $\gamma=0.8$, $v=0$ and $b_i =2$ with $\# l = 7$. Hence, in these cases either the RoA filter ($v=0.9$) or the distance filter ($\gamma=0.8$) hinders convergence due to not fulfilled Assumption~\ref{assumption:RoA_approximation} or Remark~\ref{remark:tuning_parameter_gamma}. Noticeably, for $\gamma=0.8$ with $v=0$ or $v=0.7$ and for $\gamma=0.6$ with $v=0.9$, still $J_{\ISOC}^{\max}<-0.98$\comment{this shows some kind of robustness property of the algorithm; if the assumption on sample point in $\mathcal{M}^{J_{\ISOC}}(m^*)$ cannot be fulfilled, perhaps a slightly relaxed assumption on a near global minimum (see comment before) is fulfilled first before solutions get completely bad; this near global optimal solution yields perhaps still good enough performance regarding model predictions and so on}. Overall, the best tuning parameters for our example system are $v\leq 0.7$ and $\gamma=0.6-0.7$ independent from the choice of upper bounds $b_i$ ($\forall i \in \{1,\dots,24\}$).
\begin{table}[t]
	\centering
	\caption{Tuning parameter analysis for Algorithm~\ref{algorithm:TRLwARoA}.}
	\label{table:tuningParameterRL}
	\renewcommand{\arraystretch}{1.15}
	\begin{tabular}{ |c|c|c|c|c|c| }
		\hline
		\multicolumn{3}{|c|}{} & $\gamma=0.6$ & $\gamma=0.7$ & $\gamma=0.8$ \\
		\hline
		\parbox[t]{2mm}{\multirow{6}{*}{\rotatebox[origin=c]{90}{$v=0$}}} & \multirow{3}{*}{$b_i = 2$} & $\# \Upsilon^{\mean}$ & $719.8$ & $96.3$ & $24.6$ \\
		\cline{3-6}
		& & $\# l$ & $\boldsymbol{10}$ & $\boldsymbol{9}$ & $\boldsymbol{7}$ \\
		\cline{3-6}
		& & $t_{\comp}^{\mean}$ & $258.2\,\min$ & $42.1\,\min$ & $24.1\,\min$ \\
		\cline{2-6}
		& \multirow{3}{*}{$b_i = 20$} & $\# \Upsilon^{\mean}$ & $680.7$ & $86.4$ & $24.2$ \\
		\cline{3-6}
		& & $\# l$ & $\boldsymbol{10}$ & $\boldsymbol{8}$ & $\boldsymbol{3}$ \\
		\cline{3-6}
		& & $t_{\comp}^{\mean}$ & $313.2\,\min$ & $78.2\,\min$ & $48.9\,\min$ \\
		\hline
		\parbox[t]{2mm}{\multirow{6}{*}{\rotatebox[origin=c]{90}{$v=0.7$}}} & \multirow{3}{*}{$b_i = 2$} & $\# \Upsilon^{\mean}$ & $386.9$ & $69.1$ & $20.2$ \\
		\cline{3-6}
		& & $\# l$ & $\boldsymbol{10}$ & $\boldsymbol{9}$ & $\boldsymbol{5}$ \\
		\cline{3-6}
		& & $t_{\comp}^{\mean}$ & $137.0\,\min$ & $36.8\,\min$ & $14.3\,\min$ \\
		\cline{2-6}
		& \multirow{3}{*}{$b_i = 20$} & $\# \Upsilon^{\mean}$ & $419.8$ & $67.3$ & $20.5$ \\
		\cline{3-6}
		& & $\# l$ & $\boldsymbol{10}$ & $\boldsymbol{7}$ & $\boldsymbol{1}$ \\
		\cline{3-6}
		& & $t_{\comp}^{\mean}$ & $173.9\,\min$ & $39.2\,\min$ & $15.6\,\min$ \\
		\hline
		\parbox[t]{2mm}{\multirow{6}{*}{\rotatebox[origin=c]{90}{$v=0.9$}}} & \multirow{3}{*}{$b_i = 2$} & $\# \Upsilon^{\mean}$ & $33.3$ & $18.1$ & $9.1$ \\
		\cline{3-6}
		& & $\# l$ & $\boldsymbol{3}$ & $\boldsymbol{1}$ & $\boldsymbol{1}$ \\
		\cline{3-6}
		& & $t_{\comp}^{\mean}$ & $15.8\,\min$ & $14.4\,\min$ & $9.1\,\min$ \\
		\cline{2-6}
		& \multirow{3}{*}{$b_i = 20$} & $\# \Upsilon^{\mean}$ & $33.4$ & $13.3$ & $9.4$ \\
		\cline{3-6}
		& & $\# l$ & $\boldsymbol{0}$ & $\boldsymbol{1}$ & $\boldsymbol{2}$ \\
		\cline{3-6}
		& & $t_{\comp}^{\mean}$ & $17.9\,\min$ & $9.4\,\min$ & $8.0\,\min$ \\
		\hline
	\end{tabular}
\end{table}
\comment{due to implementation, sampling time as well as initial $J_{\ISOC}$ evaluation for all samples not included in TRL computation time; however, this only increases overall computation time by around $30\,\text{s}$ in all cases - integrate it to complete computation time in future}

\subsection{Comparison to State-of-the-Art-Algorithms} \label{subsec:SoA_comparison}

Table~\ref{table:SoAComparison} shows the comparison of the TRLwARoA ISOC algorithm to different SoA methods. First, the results of our grid-search (GS) and bi-level-based ISOC algorithm \cite{Karg.2023a} are included and second, we designed an additional new bi-level-based ISOC algorithm where the upper level optimization problem is solved by the Matlab implementation of the algorithm proposed in \cite{Ugray.2007}. Here, promising starting points for local solver runs in $\mathcal{U}$ are heuristically determined by a scatter search (ScS) algorithm. We set $k_{\max} = 10000$, $\gamma=0.7$ and $v=0$ for the TRLwARoA and $k_{\max}=5000$ (best performance of ScS) for the ScS algorithm. The parameters for the GS-based method as well as the upper bounds of the feasible set~$\mathcal{U}$ are the same as in \cite{Karg.2023a}. For the stochastic algorithms, TRLwARoA and ScS, $10$ runs were performed. Table~\ref{table:SoAComparison} highlights that only the TRLwARoA algorithm yields convergence in all $10$ runs. Moreover, it computes ISOC solutions ca. $3$ and $33$ times faster than the ScS and GS, respectively. Finally, Fig.~\ref{fig:trajectoriesRLGSSoA} illustrates the match with the GT data achieved by the parameters~$\bm{\theta}^{\bigtriangleup}$ identified with the TRLwARoA algorithm ($\bm{\theta}^{\bigtriangleup}$ with $J_{\ISOC}^{\max}$ chosen). 
\begin{table}[t]
	\centering
	\caption{Comparison of the TRLwARoA algorithm to the grid-search-based (GS) method \cite{Karg.2023a} and the algorithm based on the scatter search (ScS) procedure in \cite{Ugray.2007}.}
	\label{table:SoAComparison}
	\renewcommand{\arraystretch}{1.15}
	\begin{tabular}{|c|c|c|c|}
		\hline
		& TRLwARoA & ScS & GS \cite{Karg.2023a} \\
		\hline
		$\# l$ & $\boldsymbol{10}$ & $\boldsymbol{8}$ & $-$\\
		\hline
		$J_{\ISOC}^{\max}$ & $-0.9993$ & $-0.9986$ & $-0.9987$ \\
		\hline
		$t_{\comp}^{\mean}$ & $36.6\,\min$ & $114.2\,\min$ & $1194.0\,\min$\\
		\hline
	\end{tabular}
\end{table}
\comment{future: repeat GS algorithm on same hardware for more cleaner results}

\begin{figure}[t]
	\centering
	\includegraphics[width=3in]{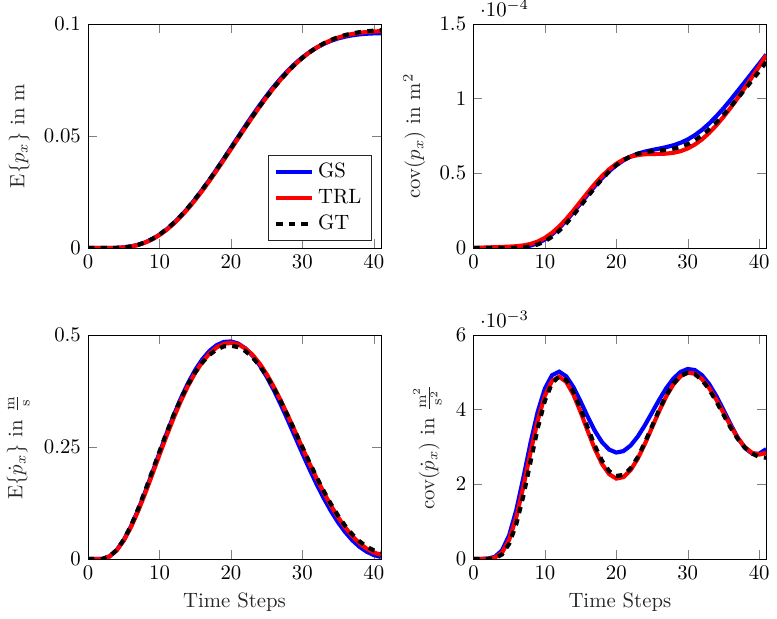}
	\caption{Mean and covariance of position~$p_x$ and velocity~$\dot{p}_x$ of the human hand in the example system. Values achieved with GT parameters and parameters identified with the TRLwARoA algorithm and the method in \cite{Karg.2023a} are shown.}
	\label{fig:trajectoriesRLGSSoA}
\end{figure}


\section{Conclusion} \label{sec:conclusion}
In this paper, we present a new algorithm to solve the ISOC problem for the LQS control model. The ISOC problem consists of determining cost function and noise scaling parameters from measurement (GT) data. Both parameter types influence the model predictions (statistical moments of the system quantities) of the LQS model. We overcome convergence problems and high computation times of our previous method \cite{Karg.2023a} by proving global convergence and achieving computation times that are nearly $33$ times smaller. Based on simulations, the tuning parameters of our new algorithm are analyzed and shown to be intuitively tuneable. Furthermore, the theoretical assumptions are validated numerically. In the future, we are going to apply our newly developed algorithm to real human measurement data. 


\bibliographystyle{IEEEtran}

\bibliography{IEEEabrv,References3}

\begin{thebibliography}{10}
\providecommand{\url}[1]{#1}
\csname url@samestyle\endcsname
\providecommand{\newblock}{\relax}
\providecommand{\bibinfo}[2]{#2}
\providecommand{\BIBentrySTDinterwordspacing}{\spaceskip=0pt\relax}
\providecommand{\BIBentryALTinterwordstretchfactor}{4}
\providecommand{\BIBentryALTinterwordspacing}{\spaceskip=\fontdimen2\font plus
\BIBentryALTinterwordstretchfactor\fontdimen3\font minus
  \fontdimen4\font\relax}
\providecommand{\BIBforeignlanguage}[2]{{%
\expandafter\ifx\csname l@#1\endcsname\relax
\typeout{** WARNING: IEEEtran.bst: No hyphenation pattern has been}%
\typeout{** loaded for the language `#1'. Using the pattern for}%
\typeout{** the default language instead.}%
\else
\language=\csname l@#1\endcsname
\fi
#2}}
\providecommand{\BIBdecl}{\relax}
\BIBdecl

\bibitem{Molloy.2022}
T.~L. Molloy, J.~Inga, S.~Hohmann, and T.~Perez, \emph{Inverse Optimal Control
  and Inverse Noncooperative Dynamic Game Theory. A Minimum Principle
  Approach}.\hskip 1em plus 0.5em minus 0.4em\relax Cham: {Springer Nature},
  2022.

\bibitem{Todorov.2002}
E.~Todorov and M.~I. Jordan, ``Optimal feedback control as a theory of motor
  coordination,'' \emph{Nat. Neurosci.}, vol.~5, no.~11, pp. 1226--1235, 2002.

\bibitem{Uno.1989}
Y.~Uno, M.~Kawato, and R.~Suzuki, ``Formation and control of optimal trajectory
  in human multijoint arm movement,'' \emph{Biol. Cybern.}, vol.~61, pp.
  89--101, 1989.

\bibitem{Jin.2019}
W.~Jin, D.~Kulic, J.~F.-S. Lin, S.~Mou, and S.~Hirche, ``Inverse optimal
  control for multiphase cost functions,'' \emph{IEEE Trans. Rob.}, vol.~35,
  no.~6, pp. 1387--1398, 2019.

\bibitem{Oguz.2018}
O.~S. Oguz, Z.~Zhou, S.~Glasauer, and D.~Wollherr, ``An inverse optimal control
  approach to explain human arm reaching control based on multiple internal
  models,'' \emph{Sci. Rep.}, vol.~8, no.~1, 2018.

\bibitem{Berret.2011}
B.~Berret, E.~Chiovetto, F.~Nori, and T.~Pozzo, ``Evidence for composite cost
  functions in arm movement planning: an inverse optimal control approach,''
  \emph{PLoS Comput. Biol.}, vol.~7, no.~10, 2011.

\bibitem{Gallivan.2018}
J.~P. Gallivan, C.~S. Chapman, D.~M. Wolpert, and J.~R. Flanagan,
  ``Decision-making in sensorimotor control,'' \emph{Nat. Rev. Neurosci.},
  vol.~19, no.~9, pp. 519--534, 2018.

\bibitem{Todorov.2005}
E.~Todorov, ``Stochastic optimal control and estimation methods adapted to the
  noise characteristics of the sensorimotor system,'' \emph{Neural Comput.},
  vol.~17, pp. 1084--1108, 2005.

\bibitem{Karg.2023b}
P.~Karg, S.~Stoll, S.~Rothfuß, and S.~Hohmann, ``Validation of stochastic
  optimal control models for goal-directed human movements on the example of
  human driving behavior,'' \emph{IFAC World Congress 2023}, 2023.

\bibitem{Karg.2023a}
------, ``Inverse stochastic optimal control for linear-quadratic gaussian and
  linear-quadratic sensorimotor control models,'' \emph{61st IEEE Conf. Decis.
  Control}, 2023.

\bibitem{Priess.2014}
M.~C. Priess, J.~Choi, and C.~Radcliffe, ``The inverse problem of
  continuous-time linear quadratic gaussian control with application to
  biological systems analysis,'' \emph{ASME 2014 Dynamic Systems and Control
  Conference}, 2014.

\bibitem{Kolekar.2018}
S.~Kolekar, W.~Mugge, and D.~Abbink, ``Modeling intradriver steering
  variability based on sensorimotor control theories,'' \emph{IEEE Trans.
  Hum.-Mach. Syst.}, vol.~48, no.~3, pp. 291--303, 2018.

\bibitem{Liang.2023}
X.~Liang, A.~Zhang, Z.~Liu, Y.~Du, and J.~Qiu, ``Suboptimal linear output
  feedback control of discrete-time systems with multiplicative noises,''
  \emph{IEEE Trans. Circuits Syst. II, Exp. Briefs}, vol.~70, no.~4, pp.
  1480--1484, 2023.

\bibitem{Moore.1999}
J.~B. Moore, X.~Y. Zhou, and A.~E.~B. Lim, ``Discrete time lqg controls with
  control dependent noise,'' \emph{Syst. Control Lett.}, vol.~36, pp. 199--206,
  1999.

\bibitem{Joshi.1976}
S.~Joshi, ``On optimal control of linear systems in the presence of
  multiplicative noise,'' \emph{IEEE Trans. Aerosp. Electron. Syst.}, vol.~12,
  no.~1, pp. 80--85, 1976.

\bibitem{Astrom.1970}
K.~J. Astr{\"o}m, \emph{Introduction to Stochastic Control Theory}.\hskip 1em
  plus 0.5em minus 0.4em\relax New York: {Academic Press, Inc.}, 1970.

\bibitem{Bezdek.2002}
J.~C. Bezdek and R.~J. Hathaway, ``Some notes on alternating optimization,''
  \emph{AFSS International Conference on Fuzzy Systems}, 2002.

\bibitem{Mombaur.2010}
K.~Mombaur, A.~Truong, and J.-P. Laumond, ``From human to humanoid
  locomotion---an inverse optimal control approach,'' \emph{Auton. Robots},
  vol.~28, no.~3, pp. 369--383, 2010.

\bibitem{Locatelli.1999}
M.~Locatelli and F.~Schoen, ``Random linkage: a family of acceptance/rejection
  algorithms for global optimisation,'' \emph{Math. Program.}, vol.~85, no.~2,
  pp. 379--396, 1999.

\bibitem{RinnooyKan.1987b}
A.~H.~G. {Rinnooy Kan} and G.~T. Timmer, ``Stochastic global optimization
  methods part ii: Multi level methods,'' \emph{Math. Program.}, vol.~39,
  no.~1, pp. 57--78, 1987.

\bibitem{Nocedal.2006}
J.~Nocedal and S.~J. Wright, \emph{Numerical Optimization}, 2nd~ed.\hskip 1em
  plus 0.5em minus 0.4em\relax New York: Springer, 2006.

\bibitem{RinnooyKan.1987a}
A.~H.~G. {Rinnooy Kan} and G.~T. Timmer, ``Stochastic global optimization
  methods part i: Clustering methods,'' \emph{Math. Program.}, vol.~39, no.~1,
  pp. 27--56, 1987.

\bibitem{Beghi.1998}
A.~Beghi, ``Discrete-time optimal control with control-dependent noise and
  generalized riccati difference equations,'' \emph{Automatica}, vol.~34,
  no.~8, pp. 1031--1034, 1998.

\bibitem{Ugray.2007}
Z.~Ugray, L.~Lasdon, J.~Plummer, F.~Glover, J.~Kelly, and R.~Mart{\'i},
  ``Scatter search and local nlp solvers: A multistart framework for global
  optimization,'' \emph{INFORMS J. Comput.}, vol.~19, no.~3, pp. 328--340,
  2007.

\end{thebibliography}

\end{document}